\documentclass[12pt, reqno]{amsart}
\usepackage{amsmath}
\usepackage{amsfonts}
\usepackage{amssymb}
\usepackage{graphicx}
\usepackage{color}
\usepackage{verbatim}

\usepackage{amsthm}
\usepackage[left=2.9cm,right=2.9cm,top=3cm,bottom=3cm]{geometry}

\newtheorem{theorem}{Theorem}

\newtheorem{corollary}[theorem]{Corollary}

\newtheorem{definition}[theorem]{Definition}
\newtheorem{example}[theorem]{Example}

\newtheorem{lemma}[theorem]{Lemma}

\newtheorem{proposition}[theorem]{Proposition}
\newtheorem{remark}[theorem]{Remark}

\title[Diagonal sections, mchr, order statistics]
{Diagonal sections of copulas,
\\multivariate conditional hazard rates
\\
and distributions of order statistics
		\\for minimally stable lifetimes}
\author[R.\,Foschi]{Rachele Foschi}
\address{University of Pisa, Department of Economics and Management,
Via C. Ridolfi, 10, 56124 Pisa, Italy}
\email{rachele.foschi@unipi.it}

\author[G.\,Nappo]{Giovanna Nappo}
\address{University of Rome La Sapienza, Department of Mathematics,
Piazzale Aldo Moro, 5, 00185, Rome, Italy}
\email{nappo@mat.uniroma1.it}

\author[F.\,L.\,Spizzichino]
 {Fabio L.\,Spizzichino}
\address{University of Rome La Sapienza,
Piazzale Aldo Moro, 5, 00185, Rome, Italy}
\email{fabio.spizzichino@fondazione.uniroma1.it}
\begin{document}

\begin{abstract}
As a motivating problem, we aim to study some special aspects of the marginal
distributions of the order statistics for exchangeable and (more generally)
for \textit{minimally stable} non-negative random variables $T_{1},...,T_{r}$.
In any case, we assume that $T_{1},...,T_{r}$ are identically distributed,
with a common survival function $\overline{G}$ and their survival copula is
denoted by $K$. The diagonal's and subdiagonals' sections of $K$, along with
$\overline{G}$, are possible tools to describe the information needed to
recover the laws of order statistics.

When attention is restricted to the absolutely continuous case, such a joint
distribution can be described in terms of the associated multivariate
conditional hazard rate (m.c.h.r.) functions. We then study the distributions
of the order statistics of $T_{1},...,T_{r}$ also in terms of the system of
the m.c.h.r.\@ functions. We compare and, in a sense, we combine the two
different approaches in order to obtain different detailed formulas and to
analyze some probabilistic aspects for the distributions of interest. This
study also leads us to compare the two cases of exchangeable and minimally
stable variables both in terms of copulas and of m.c.h.r.\@ functions. The paper
concludes with the analysis of two remarkable special cases of stochastic
dependence, namely Archimedean copulas and load sharing models. This analysis
will allow us to provide some illustrative examples, and some discussion about
peculiar aspects of our results.
\end{abstract}

\maketitle

\section{Introduction}

\label{sec:introduction}

Concerning the basic role of the concept of copula and of the Sklar's theorem
in the analysis of stochastic dependence, a main issue is the study of the
distributions of the order statistics $X_{1:r},...,X_{r:r}$ for a set of
interdependent random variables $X_{1},...,X_{r}$. On one hand, the condition
of exchangeability is specially relevant (see in particular Galambos (1982)
\cite{Galambos-1982}) in such a study. On the other hand, the marginal
distributions of $X_{1:r},...,X_{r:r}$ are strictly related to the diagonal,
and sub-diagonal, sections of copulas (see, e.g., Jaworski~(2009)~\cite{Jaworski-2009}, Durante and Sempi (2016)~\cite{Durante-Sempi-2016}).
For these reasons, in the theory of order
statistics, the study of diagonal and sub-diagonal sections of copulas has
been mainly concentrated on the case of exchangeable random variables.

Really, in such a study, the assumption of exchangeability can at any rate be
replaced by the more general condition that, for $d=2,...,r-1$, all the
$d$-dimensional diagonal sections of copulas do coincide. Such a condition has
been attracting more and more interest in the recent literature, where it has
been however designated by means of different terminologies. In fact, such a
condition can actually manifest under different mathematical forms, as we will
discuss in details. For our purposes it is specially convenient to look at it
as the condition that $X_{1},...,X_{r}$ are minimally stable (see Definition~\ref{def:min-stable} below).

In this note we concentrate attention on the case of non-negative, minimally
stable, random variables which we denote by $T_{1},...,T_{r}$.

Generally, concerning with non-negative random variables, stochastic
dependence can also be conveniently described in terms of stochastic
intensities of related counting processes. See in particular Arjas (1981)
\cite{Arjas-1981}, Bremaud (1981)~\cite{Bremaud-1981}, Arjas and Norros
\cite{Arjas-Norros-1989-91}. Such a description, in particular, can be based
on the knowledge of the so-called \textit{multivariate conditional hazard
rates (m.c.h.r.)} functions, when attention is restricted to the absolutely
continuous case (see in particular the papers by Shaked and Shanthikumar
\cite{Shaked-Shanthikumar-1990, Shaked-Shanthikumar-2007,
Shaked-Shanthikumar-2015}). In such a case the family of those functions gives
rise to a method to describe a joint distribution, which is alternative to the
one based on copulas and marginal distributions or on the joint density function.

From an analytical view-point the two methods are actually equivalent in that
a formula is clearly available to obtain in terms of the joint density the
family of the m.c.h.r.\@ functions, just in view of the very definition of them.
On the other hand, an "inverse" formula permits to recover the joint density
function when the m.c.h.r.\@ functions are known. As a matter of fact, however,
these formulas are not easily handleable in general cases. The two methods,
furthermore, are respectively apt to explain completely different aspects of
stochastic dependence.

In this paper we aim to establish a bridge between the two different
approaches. Maintaining the attention focused on the minimally stable case,
then, we are primarily interested in the relations tying the system of the
diagonal and subdiagonal sections with the system formed by the m.c.h.r.
functions. Such relations will allow us to detect, both in terms of copulas
and in terms of the m.c.h.r.\@ functions, which are the minimal sets of
functions able to convey sufficient information to recover the family of the
marginal distributions of the order statistics $T_{1:r},...,T_{r:r}$.

In such a framework, interesting questions also concern with understanding the
real difference between the cases when $T_{1},...,T_{r}$ are exchangeable and
when they are minimally stable. On this purpose, the differences between the
two properties will be detailed both using the language of copulas and the
language of the m.c.h.r.\@ functions. Still by using and combining the two
approaches, we will also face the problem of constructing examples of random
variables $T_{1},...,T_{r}$ which are minimally stable but not exchangeable.

We notice that the assumption of absolute continuity allows us to explain, in
a simpler way, the main ideas concerning the bridge between the two
approaches. However the same ideas remain substantially valid also in more
general cases, which may be treated in terms of stochastic intensities of the
counting process $\{N_{t}^{A},t\geq0\}$, where $A\subset\{1,...,r\}$ and
$N_{t}^{A}:=\sum_{i\in A}\mathbf{1}_{\{T_{i}\leq t\}}$, and in terms of
concepts related to stochastic filtering.

More in details, the plan of this paper goes as follows.

In Section 2 we introduce some needed notation and then we review basic facts
about distributions of order statistics, about diagonal sections of copulas,
and about the relations tying these two families of objects. We also show in
details the equivalence among  different forms  under which one can represent
the condition that $T_{1},...,T_{r}$ are minimally stable. Some relevant
remarks are given and an example is presented concerning the construction of
random variables which are minimally stable but not exchangeable.

In Section 3 we will first recall, in general, the definition and some basic
aspects of the family of the multivariate conditional hazard rate functions. We will
then show the special features of the cases where the lifetimes $T_{1},...,T_{r}$ are exchangeable or minimally stable. In this frame, the relations
pointed out in Section~\ref{sec:Diagonal-sections} will emerge as a natural
tool to obtain, in Section~4, the relations existing among diagonal sections
of copulas, the distributions of order statistics, and a special subclass of
multivariate conditional hazard rates. See in particular
Propositions~\ref{Talamone2-min} and~\ref{Circeo-min}.

In order to demonstrate some special aspects of the results presented in the
Sections~\ref{sec:MCHR} and~\ref{sec:Relations}, Section~\ref{sec:special}
will be devoted to a detailed discussion of the remarkable cases  of
exchangeable models based on Archimedean copulas and on minimally stable time
homogeneous load-sharing models. Some more  general examples   will be presented in the Appendix.

Often, along the paper, the term \textit{lifetime} will be used as a short-hand
for "non-negative random variable".

\textbf{Notation:}\quad For any natural number $n$, we set
$[n]=:\{1,2,...,n\}$.

For any subset $J\subseteq[n]$, we denote by $|J|$ the cardinality of $J$, and
as usual we denote by $J^{c}$ the complementary set of $J$, i.e., the set of
indices $[n]\setminus J$. Furthermore, for any $k\leq|J|$ we denote by
\[
\Pi_{k}(J)=\big\{(j_{1},...,j_{k}): j_{\ell}\in J, \, \forall\,\ell=1,..,k, \,
j_{\ell}\neq j_{h}, \, \forall\ell\neq h \big\},
\]
the set of $k$-permutations of $J$. When $k=|J|$ we drop the index $k$ and
write simply $\Pi(J)$. The symbol
\[
(n)_{k}:= n(n-1)\cdots(n-(k-1))) = | \Pi_{k}([n])|
\]
denotes the number of $k$-permutations in $\Pi_{k}([n])$. \newline

For any subset $A=\{j_{1},...,j_{\ell}\}\subset[n]$ we denote by
$\mathbf{e}_{A}$ the vector whose $i$-th component equal to 1 if $i\in A$, and
is equal to 0, otherwise.


\section{Diagonal sections and distributions of order statistics}\label{sec:Diagonal-sections}

Let $T_{1},...,T_{r}$ denote $r$ non-negative random variables, defined on a
same probability space $\left(  \Omega,\mathcal{F},\mathbb{P}\right)  $, with joint survival function
$$\overline{F}(t_{1},...,t_{r}):=\mathbb{P}\left(  T_{1}>t_{1},...,T_{r}>t_{r}\right),
$$
and survival copula $K:\left[  0,1\right]
^{r}\rightarrow\left[  0,1\right]  $.

We assume that $T_{1},...,T_{r}$ are no-ties, i.e.,
$\mathbb{P}(T_i=T_j)=0$, for $i\neq j$, and
denote by $T_{1:r},...,T_{r:r}$ the order statistics of $\left(
T_{1},...,T_{r}\right)  $ and by
\begin{equation}
\overline{G}_{1:r}(t):=\mathbb{P}\left(  T_{1:r}>t\right)  ,...,\overline
{G}_{r:r}(t):=\mathbb{P}\left(  T_{r:r}>t\right)  \label{SurvFunctnsOrdStat}
\end{equation}
the corresponding marginal survival functions.

We assume that $T_1,...,T_r$ are identically distributed, and denote by $\overline{G}$  their common one-dimensional
marginal survival function, so that
\[
\overline{G}(t):=\mathbb{P}\left(  T_{j}>t\right),\quad \text{for $j=1,...,r$, and for $t>0$.}
\]
For simplicity's sake we also assume the following condition
\begin{description}
\item[\textbf{(H)}] \qquad  $\overline{G}(t)$ is continuous, strictly positive, and strictly decreasing, $\forall t>0$.
\end{description}

Since $T_{1},...,T_{r}$ are non-negative, one has $\overline{G}(0)=1$.

\bigskip

Assume also for the moment that their joint survival function $\overline{F}(t_{1},...,t_{r})$
is exchangeable,
\[
\overline{F}(t_{1},...,t_{r})=K\left(  \overline{G}(t_{1}),...,\overline
{G}(t_{r})\right), \quad \text{for $t_{1},...,t_{r}>0$,}
\]
with $K$ permutation-invariant.

The symbol $\delta_{r}$ denotes the diagonal section of $K$, namely the
function $\delta_{r}:\left[  0,1\right]  \rightarrow\left[  0,1\right]$
defined as the trace of $K$ on the diagonal:
\begin{equation}
\delta_{r}(u):=K(u,...,u),\forall u\in\left[  0,1\right]  .
\label{SubDiagonalSections}
\end{equation}

More in general,
 being $K$ permutation-invariant,
 the functions $\delta_{\ell}:\left[
0,1\right]  \rightarrow\left[  0,1\right]  $ are defined as the traces of $K$
on the sub-diagonals of the different dimensions: for $\ell=2,...,r-1$,
\begin{equation}\label{def:delta-ell}
\delta_{\ell}(u):=K(\overbrace{u,...,u}^{\ell\text{ times}},\overbrace{1,...,1}^{r-\ell\text{
times}}),
\end{equation}

An almost obvious, but important, remark is that the sub-diagonal $\delta_{\ell}(u)$ coincides with the diagonal of the exchangeable copula which is
obtained as the $\ell$-dimensional margin of $K$. Therefore we will refer  to the functions $\delta_\ell(u)$ also as diagonal sections.
\\
It is clear that $\delta_{\ell}(u)$ is an increasing function and that
$\delta_{2}(u)\geq\delta_{3}(u)\geq...\geq\delta_{r}(u)$. Conditions, for a
function $\delta:\left[  0,1\right]  \rightarrow\left[  0,1\right]  $ to be
the diagonal section of a copula, are given, in particular, in   Jaworski (2009) \cite{Jaworski-2009}, and  Durante and Sempi \cite{Durante-Sempi-2016}.

As well-known, a direct relationship can be established between $\delta_{r}$
and the probability law of the minimal order statistics $T_{1:r}$, in fact one
immediately obtains, for $t>0$,
\begin{equation}
\overline{G}_{1:r}(t)=\mathbb{P}\left(  T_{1}>t,...,T_{r}>t\right)
=\delta_{r}(\overline{G}(t)). \label{Fregene}
\end{equation}
By taking into account  exchangeability of $T_1,...,T_r$
one can similarly write
\begin{align}\notag
\mathbb{P}\left(  T_{j_{1}}>t,...,T_{j_{d}}>t\right)  &=\mathbb{P}\left(
T_{1}>t,...,T_{d}>t\right)  ,
\\
=\overline{F}(t,...,t,0,...,0)&=K\left(  \overline{G}(t),...,\overline
{G}(t),1,...,1\right)  =\delta_{d}(\overline{G}(t)), \label{SantaSevera}
\end{align}
for $d=1,2,...,r-1$ and for any subset of indices $J=\{j_1,..,j_d\}\subset [r]$ of cardinality $d$.
Whence one can write
\begin{equation}
\overline{G}_{\ell:r}(t)=\sum_{h=r-\ell+1}^{r}\left(  -1\right)  ^{h-r-1+\ell}\binom
{r}{h}\binom{h-1}{r-\ell}\delta_{h}(\overline{G}(t)),\quad \ell=1,...,r.
\label{OrdStatEtDiagonals}
\end{equation}

In fact, by using~\eqref{SantaSevera}, the latter formula is readily obtained
from the formula expressing the survival functions of the order statistics of
exchangeable variables in terms of the survival functions of the minima within
subsets of the same variables (see in particular David, Nagaraja (2003), p.~46 \cite{David-Nagaraja-2003}, Jaworski and Rychlik (2008)~\cite{Jaworski-Rychlik-2008},\ Rychlik (2010)~\cite{Rychlik-2010}).

 We emphasize at this point that formula~\eqref{OrdStatEtDiagonals}
  for $\overline{G}_{\ell:r}(t)$ is still valid when
the joint distribution of $T_{1},...,T_{r}$  satisfy the specific symmetry conditions recalled in Definitions 1 and 2, below. Such conditions are actually weaker than exchangeability,   and  turn out to be equivalent  each other (see
Proposition~\ref{prop:equivalent-properties}  below).

\begin{definition}\label{def:t-EX}
 We will say that the random variables $T_1,...,T_r$ are \emph{$t$-Exchangeable} if for every $t\geq 0$, the binary random variables $X_i(t)=\mathbf{1}_{\{T_i>t\}}$, $i=1,...,r$, are exchangeable, or equivalently the events $\{T_i>t\}$, $i=1,...,r$, are exchangeable.
 \end{definition}
 We will briefly refer to the previous property as $t$-EX.
		\\

\begin{definition}\label{def:min-stable}
The random variables $T_1,...,T_r$  are said \emph{minimally stable}, when, for any $\ell=1,...,r$ and for any subset $A=\{j_1,...,j_\ell\}\subseteq [r]$
\begin{align}\label{eq:min-stable}
&\mathbb{P}\left(  T_{j_{1}}>t,...,T_{j_{\ell}}>t\right)  =\mathbb{P}\left(
T_{1}>t,...,T_{\ell}>t\right), \quad \forall \,t>0,
\end{align}
namely
$\mathbb{P}\left(  T_{j_{1}}>t,...,T_{j_{\ell}}>t\right)  =\overline{F}(\underbrace{t,...,t}_{\ell\ \text{times}},\underbrace{0,...,0}_{r-\ell\ \text{times}}),  \, \forall \,t>0.$\\
\end{definition}

\noindent Finally we recall the strictly related concept of  diagonal dependent copulas (see Navarro, Fer\-nan\-dez-San\-chez (2009) \cite{Navarro-Fernandez-Sanchez-2020}). This concept is related to $k$-diagonal dependence, for $k\leq r$, introduced in Okolewski (2017)\cite{Okolewski-2017}.
To this purpose for a $r$-dimensional copula $C$, and  for any $A\subset [r]$ we set
\begin{equation*}
\delta_A^C(u):=C(u\mathbf{e}_A+ \mathbf{e}_{A^c}).
\end{equation*}
Note that $\delta_A^C(u)$ is the diagonal section of the marginal copula, corresponding to the $A$-components.  When $A=[\ell]$ we set
\begin{equation*}
\delta_\ell^C(u):=\delta_{[\ell]}^C(u)= C(\overbrace{u,...,u}^{\ell\text{ times}},\overbrace{1,...,1}^{r-\ell\text{
times}})
\end{equation*}
\begin{definition}\label{def:k-DD-and-DD} Let $C$ be an $r$-dimensional copula $C$.
The copula $C$
  is
said to be a $k$-\emph{diagonal dependent} copula, with $k\leq r$, if for any $\ell\leq k$ and
 for any subset $A=\{j_1,...,j_\ell\}\subset [r]$
\begin{equation}\label{def:ell-DD}
\delta_{A}^C(u)=\delta_\ell^C(u).
\end{equation}
When $k=r$, the copula $C$ is said  \emph{diagonal dependent}.
\end{definition}

 As in Navarro and Fernandez-Sanchez (2020) \cite{Navarro-Fernandez-Sanchez-2020} we briefly refer to  the property of diagonal dependence as DD.

 The following result can be obtained by taking into account basic
and well known properties of exchangeable binary random variables originally obtained by de Finetti (see \cite{deFinetti-1937}). See also Navarro et al. (2021) \cite{Navarro-Rychlik-Spizzichino-2021}.

\begin{proposition}\label{prop:equivalent-properties}
The following properties are equivalent

\noindent\textbf{(i)}\;
The  random variables $T_1,...,T_r$,  are $t$-Exchangeable;

\noindent\textbf{(ii)} For all  $H,H^\prime \subseteq\{1,2,...,r\}$, with $|H|=|H^\prime|$
\begin{equation}\label{eq:cond-H-H'-uguali}
\mathbb{P}\big(T_j>t,  \forall j\in H,  T_i\leq t,  \forall i \notin H  \big) =  \mathbb{P}\big(T_j>t,  \forall j\in H^\prime,  T_i\leq t,  \forall i \notin H^\prime  \big).
\end{equation}

\noindent\textbf{(iii)} The random variable $T_1,...,T_r$ are minimally stable;

\noindent\textbf{(iv)} The random variable $T_1,...,T_r$ are identically distributed and their copula $K$ is diagonal dependent.

\end{proposition}
\begin{proof}
Properties \emph{\textbf{(i)}} and \emph{\textbf{(ii)}} are clearly equivalent: indeed
$$
\mathbb{P}\big(T_j>t,  \forall j\in H,  T_i\leq t,  \forall i \notin H  \big) =\mathbb{P}\big(X_i(t)=1, \forall j\in H,  X_i(t)=0,  \forall i \notin H  \big).
$$
Similarly properties \emph{\textbf{(iii)}} and \emph{\textbf{(iv)}} are equivalent:
indeed if $T_1,...,T_r$ are minimally stable, then by taking $\ell=1$ in~\eqref{eq:min-stable}, they are identically distributed, and therefore   for all $A\subset [r]$ with $|A|=\ell$
$$
\mathbb{P}(T_i>t, \, \forall i\in A)=\mathbb{P}( K(\overline{G}(t)\mathbf{e}_A+ \mathbf{e}_{A^c})=\delta_\ell(\overline{G}(t)), \quad \forall \, t\geq 0,
$$
and $\overline{G}(t)$ is invertible, in view of  the regularity condition \textbf{(H)}.
\\
Finally \emph{\textbf{(iv)}} is equivalent to \emph{\textbf{(ii)}}, in view of the inclusion-exclusion formula.
\end{proof}
	\begin{remark}\label{rem:letteratura}
The previous result (Proposition~\ref{prop:equivalent-properties})  holds true also without the assumption~$\emph{\textbf{(H)}}$,
but in the general case  an extension   of the notion of Diagonal Dependence is needed.
 (see Navarro et al.\@ (2021) \cite{Navarro-Rychlik-Spizzichino-2021}).
 	\end{remark}
 The interest for the properties \emph{\textbf{(i)}} and \emph{\textbf{(iv)}} had  independently emerged  in the two papers Marichal et al.~(2011)~\cite{Marichal-et-al-2011} and Navarro and Fernandez-Sanchez (2020)~\cite{Navarro-Fernandez-Sanchez-2020} with reference to the field of systems' reliability. Still in the same framework, furthermore, the study of conditions for the equivalence between \emph{\textbf{(i)}} and \emph{\textbf{(iv)}} has been developed in Navarro et al.~(2021)~\cite{Navarro-Rychlik-Spizzichino-2021}.
    More precisely the above mentioned articles deal with the so-called signature representation for the survival function $R_{S}(t)$ of the lifetime $T_{S}$ of a binary coherent system $S$. For a coherent system made with $r$ binary components, the signature is a probability distribution $\mathbf{s}:=(s_1,...,s_{r})$ over $[r]$ which is a combinatorial invariant associated to the structure function $\varphi$ of $S$ (see in particular Samaniego~(2007)~\cite{Samaniego-2007}). The signature representation means that the equation
\begin{equation}\label{R-S}
R_{S}^{(\varphi)}(t)=\sum_{h=1}^rs_{h}\overline{G}_{h:r}(t)
\end{equation}
holds for any $t>0$, where $\overline{G}_{1:r}$,...,$\overline{G}_{r:r}$ denote the survival functions of the order statistics of the components' lifetimes. See also the Remark~\ref{rem:reliability}  below.

We are now in a position to establish the following result.

\begin{proposition}\label{prop:min-stable-Ord-Stat}The equations~\eqref{OrdStatEtDiagonals} hold, under any of the conditions \textbf{(i)}\---\textbf{(iv)}.
\end{proposition}

Actually the validity of~\eqref{OrdStatEtDiagonals} hinges on the Eq.~\eqref{SantaSevera}, which only requires the DD property  of the survival copula and the identical distribution of $T_i$, $i=1,...,r$.

From now on we assume that the random variables $T_1,...,T_r$ are minimally stable.

Since the marginal survival functions $\overline{G}_{1:r},...,\overline{G}_{r:r}$ are determined by the knowledge of the joint distribution of
$T_{1},...,T_{r}$ then, in principle, they should depend on the survival
copula $K$. As shown by the formula~\eqref{OrdStatEtDiagonals},
however, $\overline{G}_{1:r},...,\overline{G}_{r:r}$ are actually
determined by the only knowledge of the diagonal sections $\delta_{2},...,\delta_{r}$ and the common survival function $\overline{G}$.
\\

On the other hand, when
 $\overline{G}_{1:r},...,\overline{G}_{r:r}$ are known,
we can easily recover the common marginal survival function $\overline{G}(t)$ and  the functions $\delta_{2},...,\delta_{r}$. Indeed,
 the random variables $T_1$,...,$T_r$  are identically distributed  and therefore
\begin{align}\label{bar-G-con-bar-G-k-r}
\overline{G}(t)= \frac{1}{r} \sum_{k=1}^r \overline{G}_{k:r}(t),
\end{align}
as immediately follows by observing that $\sum_{h=1}^r \mathbf{1}_{\{T_h>t\}}= \sum_{k=1}^r \mathbf{1}_{\{T_{k:r}>t\}}$.
Furthermore the same
 formula~\eqref{OrdStatEtDiagonals} permits to
recover, step-by-step, the functions $\delta_{2},...,\delta_{r}$.  A detailed formula is given in the Corollary~\ref{cor:delta-da-G-k} of the following result.

\begin{proposition}\label{prop:delta-da-G-k}
Let $T_1,T_2,...,T_r$ be minimally stable, with order statistics distributed according to   $\overline{G}_{1:r},\,\overline{G}_{2:r},...,\,\overline{G}_{r:r}$.
\\
  Then, for every $d\in\{1,2,...,r\}$, and $J\subseteq \{1,2,..,r\}$, with $|J|=d$
\begin{align}\label{prob-surv-min-on-J-a}
\mathbb{P}\big(T_j>t, \, \forall j\in J  \big)
&= \sum_{h=d}^{r}\frac{(h)_d}{(r)_d}\,  \Big(\overline{G}_{r-h+1:r}(t)- \overline{G}_{r-h:r}(t)\Big)
\\&
=\frac{d}{(r)_d}\,\sum_{k=1}^{r-d+1} (r-k)_{d-1}\,\overline{G}_{k:r}(t), \quad t>0, \label{prob-surv-min-on-J}
\end{align}
where    by convention  $\overline{G}_{0:r}(t)=0$, for $t> 0$, and $\binom{0}{0}=1$.
\end{proposition}

Before giving the proof we get the expression of~$\delta_d$ in terms of  the marginal survival functions $\overline{G}_{k:r}(t)$, $k=1,...,r$. Indeed, this is a consequence of equations~\eqref{prob-surv-min-on-J} and~\eqref{SantaSevera} (as already observed the condition appearing in the latter equation also holds for minimally stable variables).

\begin{corollary}\label{cor:delta-da-G-k}
Assume   the same conditions  of Proposition~\ref{prop:delta-da-G-k} and
 condition~\textbf{\emph{(H)}}. Then
   for every $d\in\{1,2,...,r\}$, the following equalities hold
\begin{align}\label{delta-from-G-0}
\delta_d(u)&
=\sum_{h=d}^{r}  \frac{(h)_d}{(r)_d}\,\Big(\overline{G}_{r-h+1:r}\big(\overline{G}^{-1}(u)\big)- \overline{G}_{r-h:r}\big(\overline{G}^{-1}(u)\big)\Big)
\\&=\frac{d}{(r)_d}\,\sum_{h=1}^{r-d+1} (r-h)_{d-1}\,\overline{G}_{h:r}\big(\overline{G}^{-1}(u)\big), \quad u\in  [0,1].
\label{delta-from-G}
\end{align}

\end{corollary}
\begin{remark}\label{rem:equiv-inform}
Taking into account  also~\eqref{bar-G-con-bar-G-k-r},
we can thus conclude that the family of functions $\big\{ \overline{G}_{1:r},...,\overline
{G}_{r:r}\big\} $ and the family $\big\{ \overline{G},\delta_{1},...,\delta_{r}\big\}$ convey the same piece of information.
\\
\end{remark}
\begin{remark}\label{rem:reliability}
    From the equivalence between signature representation and properties \textbf{(i)} and \textbf{(iv)} we can conclude that,  under our conditions, the same information contained in $\big\{\overline{G}_{1:r},...,\overline{G}_{r:r}\big\}$ is also contained in the family of all the reliability functions
    $R_S^{(\varphi)}(t)$ in \eqref{R-S}. We emphasize that such a family is indexed by all the coherent structures $\varphi$ for binary systems made of $r$ components.
\end{remark}

\bigskip

As far as the proof of Proposition~\ref{prop:delta-da-G-k}  is concerned, one could apply  well-known and simple results (see, e.g., de Finetti (1937) \cite{deFinetti-1937}) about exchangeable events.

The above  formulas~\eqref{prob-surv-min-on-J-a}\--- \eqref{delta-from-G} in Proposition~\ref{prop:delta-da-G-k} and  Corollary~\ref{cor:delta-da-G-k}, may also be obtained in terms of the M\"{o}bius transform starting from  formula~\eqref{OrdStatEtDiagonals}.

 For the ease of the reader we give however a self-contained, and detailed, proof. One important ingredient of the proof is  the observation that, for any subset $J\subset [r]$ one has
\begin{align}\label{eq:relazione-tra-min_suJ-e-cond-ii}
\mathbb{P}\big(T_j>t, \, \forall j\in J  \big)
&= \sum_{K:\,K\subseteq J^c}
\mathbb{P}\big(T_j>t, \, \forall j\in J \cup K, \; T_i\leq t, \, \forall i \notin J\cup K  \big)
\intertext{or equivalently}
\mathbb{P}\big(T_j>t, \, \forall j\in J  \big)&=  \sum_{H:\,H\supseteq J} \label{eq:relazione-tra-min_suJ-e-cond-ii-BIS}
\mathbb{P}\big(T_j>t, \, \forall j\in H, \; T_i\leq t, \, \forall i \notin H  \big).
\end{align}

\begin{proof}[Proof of Proposition~\ref{prop:delta-da-G-k}]
We start by observing that
\begin{align}\notag
&\mathbb{P}(N(t)=r-h)=\sum_{J: |J|=h} \mathbb{P}(T_j>t, \, \forall j\in J, \; T_i\leq t,\, \forall\, i\notin J)
\intertext{so that, thanks to Eq.~\eqref{eq:cond-H-H'-uguali}}
\notag
&\mathbb{P}(N(t)=r-h)=\binom{r}{h}\mathbb{P}(T_j>t, \, \forall j\in\{1,2,..,h\}, \;T_i\leq t,\; \forall i\in\{h+1,...,r\}),
\intertext{or equivalently, for any $H\subset [r]$, with $|H|=h$}
\label{eq:a}
&\mathbb{P}\big(T_j>t, \, \forall j\in H, \; T_i\leq t, \, \forall i \notin H  \big)= \frac{1}{\binom{r}{h}}\, \mathbb{P}(N(t)=r-h).
\intertext{On the other hand, we observe that}
&\notag
\mathbb{P}(N(t)=r-h)=
\mathbb{P}(T_{r-h:r}\leq t< T_{r-h+1:r})
\\
=& \mathbb{P}(T_{r-h+1:r}>t)-\mathbb{P}(T_{r-h:r}>t)
=\overline{G}_{r-h+1:r}(t)-\overline{G}_{r-h:r}(t).
\label{eq:b}
\end{align}
Then the thesis follows immediately: indeed, for every $J\subset \{1,2,...,r\}$ with~$|J|=d$,  Eq.s~\eqref{eq:a} and~\eqref{eq:b}, together with~\eqref{eq:relazione-tra-min_suJ-e-cond-ii},  imply
\begin{align*}
\mathbb{P}\big(T_j>t, \, \forall j\in J  \big)
&= \sum_{h=d}^{r} \binom{r-d}{h-d}\, \frac{1}{\binom{r}{h}}\,\left( \overline{G}_{r-h+1:r}(t)-\overline{G}_{r-h:r}(t)\right), 
\end{align*}
 and  formula~\eqref{prob-surv-min-on-J-a} follows by observing that
$$\frac{\binom{r-d}{h-d}}{\binom{r}{h}}=\frac {(h)_d}{(r)_d}.
$$
Finally, from~\eqref{prob-surv-min-on-J-a},  taking into account the convention that $\overline{G}_{0:r}(t)=0$,  one obtains
\begin{align*}
(r)_d\, \mathbb{P}\big(T_j>t, \, \forall j\in J  \big)
&=\sum_{h=d}^{r}  (h)_d \overline{G}_{r-(h-1):r}(t)- \sum_{h=d}^{r-1}  (h)_d \overline{G}_{r-h:r}(t).
\\
&=\sum_{k=d-1}^{r-1}  (k+1)_d \overline{G}_{r-k:r}(t)- \sum_{h=d}^{r-1}  (h)_d \overline{G}_{r-h:r}(t).
\\
&=
 d!\, \overline{G}_{r-(d-1):r}(t) +
\sum_{k=d}^{r-1}  \left[(k+1)_d- (k)_d\right] \overline{G}_{r-k:r}(t).
\end{align*}
Therefore, by observing that
$$
(k+1)_d- (k)_d= (k)_{d-1}\,[k+1- (k-(d-1))]=(k)_{d-1} \, d,
$$ one gets
\begin{align*}
 \mathbb{P}\big(T_j>t, \, \forall j\in J  \big)=\frac{d}{(r)_d}\, \sum_{k=d-1}^{r-1}  (k)_{d-1} \overline{G}_{r-k:r}(t),
\end{align*}
Then formula~\eqref{prob-surv-min-on-J} follows  by setting $h=r-k$ in the last sum.
\end{proof}
\bigskip

\begin{remark}\label{rem:EX-con-delta=min-stable}
For non-exchangeable, but minimally  stable
lifetimes $T_1,...,T_r$ there still exist   exchangeable lifetimes  $\widetilde{T}_1,...,\widetilde{T}_r$,   such that
 $\mathbb{P}(\widetilde{T}_j>t)=\mathbb{P}(T_j>t)=\overline{G}(t)$, and the
 sub-diagonal sections $\widetilde{\delta}_\ell(u)$ of their survival   copula $\widetilde{K}$ coincide  with the sub-diagonal $\delta_\ell(u)$  of the  survival   copula $K$. Indeed   $\widetilde{K}$ may be contructed by symmetrizing~$K$:
$$
\widetilde{K}(u_1,...,u_r)=\frac{1}{r!} \sum_{\sigma\in \Pi([r])} K(u_{\sigma_1},...,u_{\sigma_r}).
$$
 The above construction  can be of help in obtaining the explicit form of $\mathbb{P}(T_j>t, \, \forall\,j\in A)$ in some special cases (see in particular Section~\ref{subsec:THLS}).
  \end{remark}

\begin{remark}\label{rem:min-stable-NOT-exchangeable}
 The problem of constructing examples of  vectors which
are not exchangeable, but still minimally stable, naturally arises. Since minimally stable variables $T_1,...,T_r$ are
identically distributed, constructing such examples is equivalent to constructing  dia\-gonal-de\-pendent  copulas, which are not exchangeable.   In the following Example~\ref{example:C-3-construction} we present a  simple path to such a construction. Other
 examples 
 can be found in Navarro, Fernandez-Sanchez (2009) \cite{Navarro-Fernandez-Sanchez-2020}, and Navarro et al.~(2020) \cite{Navarro-Rychlik-Spizzichino-2021}. See also Example~\ref{example:D+C1-C2} in the Appendix.
\end{remark}

\begin{example}\label{example:C-3-construction}
First of all we notice that, when $r=2$, then any pair of identical distributed random variables $T_1$ and $T_2$ are minimally stable, but in general is not exchangeable. Similarly, and trivially, any $2$-dimensional copula $C$ is minimally stable. Indeed $\delta^{C}_2(u)=C(u,u)$, and $\delta^{C}_1(u)=C(u,1)=C(1,u)=u$.
Starting from the copula $C$ one may define two $3$-dimensional copulas as follows
$$
C_{(1,2,3)}(u_1,u_2, u_3):= \frac{1}{3}\,\big[C(u_1, u_2)u_3+C(u_2, u_3)u_1+C(u_3, u_1) u_2\big]
$$
$$
C_{(3,2,1)}(u_1,u_2, u_3):= \frac{1}{3}\,\big[C(u_3, u_2)u_1+C(u_2, u_1)u_3+C(u_1, u_3)u_2\big]
$$
respectively obtained as the symmetric mixture over the cyclic permutations of $(1,2,3)$ and the cyclic permutations of $(3,2,1)$. Notice that when $C$ is non-exchangeable, then
$C_{(1,2,3)}$ and $C_{(3,2,1)}$ are non-exchangeable: indeed if $u,v\in (0,1)$ are such that $C(u,v)\neq C(v,u)$ then, for example,
\begin{align*}
C_{(1,2,3)}(u,v, 1)=&:= \frac{1}{3}\,\big[C(u, v)+C(v, 1)u+C(1, u)v\big]= \frac{1}{3}\,\big[C(u,v)+vu+uv)\big]
\intertext{which is clearly different from}
 C_{(1,2,3)}(v,u, 1)&:= \frac{1}{3}\,\big[C(v, u)+C(u, 1)v+C(1, v)u\big]=\frac{1}{3}\,C(v, u)+\frac{2}{3}\,uv,
\end{align*}
thus $C_{(1,2,3)}$ is non-exchangeable, though the $2$-dimensional marginal distributions are all equal, namely
$$
C_{(1,2,3)}(u,v,1)=C_{(1,2,3)}(v,1,u)=C_{(1,2,3)}(1,u,v).
$$
 By iterating the above construction,
one can also obtain a  DD,  non-exchangeable, copula which is $n$-dimensional. Details can be found in the Appendix
	(see Example~\ref{example-DD-n-copulas}).
\end{example}

\section{Multivariate conditional hazard rates}\label{sec:MCHR}

In this section we restrict attention to vectors of lifetimes with  absolute continuous joint probability
distributions. In such a case, the probabilistic properties
of the latter distributions can be alternatively described in terms of the multivariate conditional hazard rate functions.

Before concentrating attention on  minimally stable  or exchangeable lifetimes $T_{1}$,...,$T_{r}$
of our concern, we briefly recall some definitions and basic properties of
multivariate conditional hazard rate functions, more in general. On this
purpose, consider $n$ non-negative random variables $V_{1},...,V_{n}$ with an
absolutely continuous joint distribution whose joint density function is
denoted by $f_{\mathbf{V}}$.

For $k=1,...,n-1$, and for any $k$-permutation $\mathbf{j}=(j_1,...,j_k)\in\Pi_k([n])$, the symbol~$\mathbf{V}_{\mathbf{j}}$ denotes the vector of lifetimes
$(V_{j_{1}},\cdots,V_{j_{k}})$; for any subset $J\subseteq [n]$ we denote
 \begin{equation}\label{MinOrdStatInJ}
 V_{1:J}:= \min_{j\in J}V_j;
 \end{equation}
furthermore, if $\mathbf{j}\in \Pi(J)$, for
$0<v_{1}<\cdots<v_{k}\leq v$ the symbol
\begin{equation}
\mathbf{V}_{\mathbf{j}}=\mathbf{v};\quad V_{1:J^c}>v \label{DefmchrA}%
\end{equation}
briefly denotes the observation
\[
 V_{j_{1}}=v_{1},\cdots,\, V_{j_{k}}=v_{k},\quad \min_{j\in J^c}V_{j}>v.
\]
The observation
 $\mathbf{V}_{\mathbf{j}}=\mathbf{v};\; V_{1:J^c}>v$ in~\eqref{DefmchrA}
  is often called a
\textquotedblleft dynamic history".

For the given $k\geq 1$, $J\subset [n]$, with $|J|=k$, $\mathbf{j}=(j_1,...,j_k)\in \Pi(J)$, $v>0$, $0<v_{1}<\cdots<v_{k}\leq v$, $j\notin J$, the multivariate conditional hazard rates (m.c.h.r.)
function $v\mapsto\lambda_{j|\mathbf{j}}(v|v_{1},\cdots,v_{k})$ is defined by the limit
\begin{equation}
\lambda_{j|j_1,...,j_k}(v|v_{1},\cdots,v_{k}):=\lim_{\Delta\rightarrow0^+}\frac{\mathbb{P}\big(V_{j}\leq v+\Delta|\mathbf{V}_{\mathbf{j}}=\mathbf{v};V_{1:J^c}>v\big)}{\Delta}. 
 \label{DefmchrC}
\end{equation}
Furthermore, for any $j\in [n]$, the m.c.h.r.\@ function $\lambda_{j|\emptyset}(v):[0,\infty
)\rightarrow\lbrack0,\infty)$ is defined by
\begin{equation}
\lambda_{j|\emptyset}(v):=\lim_{\Delta\rightarrow0^+}\frac{
\mathbb{P}\big(V_{j}\leq v+\Delta|V_{1:n}>v\big)}{\Delta}. \label{DefmchrD}
\end{equation}

In the sequel we will  use the convention
\begin{equation}
\lambda_{j|j_1,...,j_k}(v|v_{1},\cdots,v_{k})=\lambda_{j|\emptyset}(v), \quad \text{when $k=0$.} \label{CONV-mchr-k=0}
\end{equation}

The above limits make sense in view of the assumption of absolute continuity
of the joint distribution of $V_{1},...,V_{n}$ and the m.c.h.r.\@ functions can
be seen as direct extensions of the common concept of hazard rate function of
a univariate non-negative random variable.

For the random vector $\mathbf{V}\equiv\left(  V_{1},...,V_{n}\right)  $, the
system of the m.c.h.r.\@ functions in~\eqref{DefmchrC} and~\eqref{DefmchrD} can
be computed in terms of the joint density $f_{\mathbf{V}}$. It is remarkable
the circumstance that the function $f_{\mathbf{V}}$ can be obtained from the
knowledge of the set of all the m.c.h.r.\@ functions. In fact, the following
formula holds:
\\
for $ (x_{1},...,x_{n})$, let $\mathbf{j}=(j_1,...,j_n)$  a permutation in $\Pi([n])$  such that
$$
x_{1:n}=x_{j_1}\leq x_{2:n}=x_{j_2}\leq \cdots \leq x_{n:n}=x_{j_n},
$$
then the joint density may be written as
\begin{align}\notag
&f_{\mathbf{V}}\left(x_{1},...,x_{n}\right)  =\lambda_{1}(x_{j_1}%
|\emptyset)\exp\Big\{-\int_{0}^{x_{j_1}}  \sum_{j=1}^{n}\lambda_{j|\emptyset}%
(u)\,  du\Big\}\cdot
\\\notag
&{}\quad\cdot\lambda_{j_2|j_1}(x_{j_2}|x_{j_1})\exp\Big\{-\int_{x_{j_1}}^{x_{j_2}}
\sum_{j\neq j_1}\lambda_{j|j_1}(s|x_{j_1})\,  ds\Big\}\cdot
\\\notag
&{}\quad\cdot\lambda_{j_{k+1}|j_1,...,j_k}(x_{j_{k+1}}|x_{j_1},...,x_{j_k})\cdot
\\\notag
&{}\quad \quad \qquad \qquad \cdot\exp\Big\{-\int_{x_{j_k}}^{x_{j_{k+1}}}  \sum_{j\notin \{j_1,\ldots, j_k\}}\lambda_{j|j_1,...,j_k}(u|
x_{j_1},...,x_{j_k})\,  du\Big\}\cdot
\\\notag
&{}\quad\cdots \cdots
\\ \notag
&{}\quad \cdot\lambda_{j_n|j_1,...,j_{n-1}}\left(x_{j_n}|x_{j
_1},...,x_{j_{n-1}}\right) \cdot
\\  \label{JointDensInTermsmchr}
&{}\quad \quad \qquad \qquad\cdot \exp
\Big\{-\int_{x_{j_{n-1}}}^{x_{j_n}}\lambda_{j_n|j_1,...,j_{n-1}}(u|x_{j_1},...,x_{j_{n-1}})du\Big\}.
\end{align}

Then, setting, for $k>1$, 
 \begin{equation}\label{LAMBDA}
\Lambda_{j_1,....,j_{k-1}}(u)=\sum_{j\notin\{ j_1,\ldots, j_{k-1}\}}\lambda_{j|j_1,...,j_{k-1}}(u|
x_{j_1},...,x_{j_{k-1}}), 
\end{equation}
using the convention that, when $k=1$,
\begin{align}\notag
\lambda_{j|j_1,...,j_{k-1}}(u|
x_{j_1},...,x_{j_{k-1}})&=\lambda_{j|\emptyset}(u),
\\\label{LAMBDA-k=0}
\Lambda_{j_1,....,j_{k-1}}(u)\Big|_{k=1}&=\Lambda_{\emptyset}(u)=\sum_{j\in [n]}\lambda_{j|\emptyset}(u),
\end{align}
and  the further convention that $x_{j_0}=0$,
 we can write shortly
\begin{align}\notag
&f_{\mathbf{V}}\left(x_{1},...,x_{n}\right)
=\Pi_{k=1}^n \lambda_{j_{k}|j_1,...,j_{k-1}}(x_{j_{k}}|x_{j_1},...,x_{j_{k-1}})\cdot
\\
&{}\quad \quad \qquad \qquad \cdot e^{-\int_{x_{j_{k-1}}}^{x_{j_{k}}}  \Lambda_{j_1,...,j_{k-1}}(u|
x_{j_1},...,x_{j_{k-1}})\,  du}\label{eq:f_V-short}
\end{align}

For proofs, details, and for general aspects of the m.c.h.r.\@ functions see
Shaked, Shantikumar (1990), (2007) \cite{Shaked-Shanthikumar-1990, Shaked-Shanthikumar-2007}.

See also the reviews contained within the more recent papers Shaked, Shantikumar
(2015) \cite{Shaked-Shanthikumar-2015}, Spizzichino (2019) \cite{Spizzichino-2019-Rel}

\bigskip

For our purposes it is relevant to highlight that, in particular, the
functions  $\lambda_{j|\emptyset}(v)$, for $j\in [n]$,
 are strictly
related to the marginal law of the minimal order statistics $V_{1:n}\equiv V_{1:[n]}=\min_{j=1,...,n}V_{j}.$
In this respect the following identity holds:
\begin{equation}
\mathbb{P}(V_{1:n}>v)=\exp\Big\{-\int_{0}^{v}\sum_{j=1}^{n}\lambda_{j|\emptyset}%
(s)ds\Big\}=\exp\Big\{-\int_{0}^{v}\Lambda_{\emptyset}(s)ds\Big\}. \label{LawOfMinimum}%
\end{equation}
(See, e.g., De Santis et al. (2021) \cite{DeSantis-et-al-2021}, where a
more detailed description of the probabilistic behavior of $V_{i:n}$ in terms
of $\lambda_{j|\emptyset}(v)$, for $j\in [n]$,
 is pointed out).
\\
For an arbitrary subset of indices $A\subset\left[  n\right]  $, we can also
consider the survival function of~$V_{1:A}$, the minimal order statistics among the
variables $V_{j}$, with $j\in A$.

It is important to stress that formulas similar to~\eqref{LawOfMinimum} can be obtained for
the survival function of $V_{1:A}$, provided a different set of m.c.h.r.\@
functions is appropriately considered. On this purpose we can notice that, for
$A\subset\left[  n\right]  $, the $|A|$-dimensional joint distribution of
$\{V_j, \, j\in A\}$ is absolutely continuous as well. Thus, for any $k\leq$ $|A|$ and for
any $k$-permutation $\left(  j_{1},...,j_{k}\right)  \in\Pi_{k}\left(
A\right)  $, $j\in A\backslash\{j_{1},...,j_{k}\}$, $0<v_{1}<\cdots<v_{k}\leq v$, we can
consider the m.c.h.r.\@ functions defined as follows
\begin{equation}
\lambda^A_{j|j_1,...,j_k}(v|v_{1},\cdots,v_{k}):=\lim_{\Delta\rightarrow0^+}\frac{\mathbb{P}\big(V_{j}
\leq
v+\Delta|\mathbf{V}_{\mathbf{j}}=\mathbf{v};V_{1: A\setminus \{j_1,...,j_k\}}>v\big)}{\Delta
},  \label{DefmchrC-J}
\end{equation}
and, for $j\in A$, $v>0$

\begin{equation}
\lambda^A_{j|\emptyset}(v):=\lim_{\Delta\rightarrow0^+}\frac{
\mathbb{P}\big(V_{j}\leq v+\Delta|V_{1:A}>v\big)}{\Delta} .
\label{DefmchrD-J}
\end{equation}
With the above notation,  one can write
\begin{equation}
\mathbb{P}(V_{1:A}>v)=\exp\Big\{-\int_{0}^{v}\sum_{j\in A}\lambda^A_{j|\emptyset}(s)ds\Big\}=\exp\left\{-\int_{0}^{t}\Lambda^A_{\emptyset}(s)\,ds\right\},
 \label{LawOfMinimum-A}
\end{equation}
where
$$
\Lambda^A_{\emptyset}(t):=\sum_{j\in A}\lambda^A_{j|\emptyset}(t).
$$
\medskip

In view of the equivalence of condition~\eqref{eq:cond-H-H'-uguali}  and minimal stability (i.e., conditions~\textbf{\emph{(ii)}} and~\textbf{\emph{(iii)}} of Proposition~\ref{prop:equivalent-properties}), for our purposes it is also relevant to
 express
 $$\mathbb{P}\big(V_j>t,  \forall j\in A,  V_i\leq t,  \forall i \in[n]\setminus A  \big)
 $$
 in terms of the m.c.h.r.\@ functions. To this end we set
 for any $d$-permutation $\mathbf{j}= (j_1,...,j_d)$
\begin{align}
\Psi(t;[n], \mathbf{j}) :=&\mathbb{P}\big(V_{j_1}< V_{j_2}<\cdots<V_{j_d}\leq t, V_i>t\, \forall i\notin \{j_1,...,j_d\}\big)\notag
\\\notag
&= \int_0^t du_1\int_{u_1}^t du_2 \cdots \int_{u_{d-1}}^t  du_d\, e^{-\int_{u_d}^{t} \Lambda
 _{j_{1},j_{2},...,j_{d}}(u|u_1,u_2,...,u_d)  du}\cdot
 \\
 &{}\qquad \quad
 \cdot \prod_{k=1}^d  \lambda_{j_{k}|j_{1},...,j_{k-1}}(u_k| u_1,...,u_{k-1})
e^{-\int_{u_{k-1}}^{u_k} \Lambda
_{j_{1},j_{2},...,j_{k-1}}(u|u_1,u_2,..., u_{k-1}) du}
\label{def:Psi-t-n-j}
 \end{align}

Then, for any subset $A\subset [n]$,
we can write
 \begin{align}\label{eq:P(NA=0--NAc=n-d)}
&\mathbb{P}\big(V_j>t,  \forall j\in A,  V_i\leq t,  \forall i \in[n]\setminus A  \big)=  \sum_{\mathbf{j}\in \Pi([n]\setminus A)} \,\Psi(t;[n],\mathbf{j}).
 \end{align}

In some applications we will use also the following alternative  expression for $\Psi(t;[n],\mathbf{j})$
\begin{align}\notag
	 & \Psi(t;[n],\mathbf{j}) =
	\int_{0}^{t}ds_d\int_{0}^{s_d}ds_{d-1}\cdots\int_{0}^{s_2}ds_{1} \,
	e^{-\int_{s_{d}}^{t}\Lambda_{j_1,\dots,j_d}(\tau|s_{1},\dots,s_{d})d\tau}\cdot
	\\
	& {}\qquad\qquad \qquad
	\cdot \prod_{\ell=1}^{d}\lambda_{j_{\ell}|j_1,\dots,j_{\ell-1}}(s_{\ell}|s_{1},\dots,s_{\ell-1})
	e^{-\int_{s_{\ell-1}}^{s_{\ell}}\Lambda_{j_1,\dots,j_{\ell-1}}(\tau|s_{1},\dots,s_{\ell-1})
		d\tau}\,.
	\label{Psi_xeqdiff}
\end{align}

Concerning the above arguments and coming back to our lifetimes $T_1,...,T_r$, we start with the case of exchangeable  lifetimes
$T_{1}$,...,$T_{r}$ by noticing
 that the very condition
of exchangeability leads to a remarkable simplification of notation, technical
results, and conceptual aspects.

First notice that the symmetry conditions among the different random
variables, as requested by exchangeability, imply a specially simple form
for the m.c.h.r.\@ functions. More precisely, the function $\lambda
_{j|i_1,...,i_k}(v|v_{i_{1}},\ldots,v_{i_{k}})$ cannot depend on the index $j\notin \{i_1,...,i_k\}$. Furthermore all the $k$-permutations $(i_1,...,i_k)$ are to
be considered as similar one another and thus the dependence of a m.c.h.r.\@
function w.r.t.\@ to $(i_1,...,i_k)$ is encoded in the number $k$.

For the exchangeable case, we then introduce the symbols $\mu(t|k;t_{1}%
,\ldots,t_{k}),\mu(t|0)$ with the following meaning: for any $\mathbf{j}=(j_1,...,j_k)\in \Pi(I)$
\begin{align}
\lambda_{j|j_1,...,j_k}(t|t_{1},\ldots,t_{k})    =\mu(t|k;t_{1},\ldots,t_{k}), \qquad \lambda_{j|\emptyset}(t)    =\mu(t|0).\label{ExchMCHR}
\end{align}

Using this notation, and denoting by $t_{1:r},...,t_{r:r}$ the values
$t_{1},\ldots,t_{k}$ rearranged in an increasing order, the equation
\eqref{JointDensInTermsmchr} for the joint density $f_{\mathbf{T}}$ takes the form
\begin{equation}\label{JointDensInTermsmchrEX-short}
f_{\mathbf{T}}\left(  t_{1},...,t_{r}\right)=
\Pi_{k=0}^{r-1}\, \mu(t_{k+1:r}|k;t_{1:r},...,t_{k:r})e^{-(r-k)\, \int_{t_{k:r}}^{t_{k+1:r}}\,\mu(s|k;t_{1:r},...,t_{k:r})\, ds},
\end{equation}
where we used the convention that $t_{0:k}=0$.

It is immediately seen that such a joint density $f_{\mathbf{T}}\left(
t_{1},...,t_{r}\right)  $ is exchangeable and then the following
characterization of exchangeability holds.
\begin{proposition}\label{prop:caratt-Scambiabili}
Non-negative random variables $T_{1},...,T_{r}$ with an absolutely continuous
joint distribution are exchangeable if and only if the corresponding m.c.h.r.\@
functions are of the form~\eqref{ExchMCHR}.
\end{proposition}
Coming back to the case of minimally stable lifetimes $T_{1},...,T_{r}$, we may use   Eq.~\eqref{def:Psi-t-n-j}, Eq.~\eqref{eq:P(NA=0--NAc=n-d)}, and  Proposition~\ref{prop:equivalent-properties} to get the following
 characterization of minimal stability.
\begin{proposition}\label{prop:caratt-min-stable}
Non-negative random variables $T_{1},...,T_{r}$ with an absolutely continuous
joint distribution are minimally stable if and only if the corresponding m.c.h.r.\@ satisfy the condition that
whenever $A, B \subset [r]$, with $|A|=|B|\leq r-1$, then
\begin{align}\label{cond:caratt-min-stable-short}
\sum_{\mathbf{j}\in \Pi(A)}\Psi(t;[r],\mathbf{j})= \sum_{\mathbf{j}^\prime\in \Pi(B)}\Psi(t;[r],\mathbf{j}^\prime), \quad t>0,
\end{align}
where we have used the notation~\eqref{def:Psi-t-n-j}.
\end{proposition}


\section{Relations among diagonal sections of 
DD copulas,
distributions of order statistics, and multivariate conditional hazard rates}\label{sec:Relations}

Concerning with the joint distribution of
minimally stable lifetimes
 $T_{1},...,T_{r}$, it has been pointed out in Remark \ref{rem:equiv-inform} that the two systems
of functions
\begin{itemize}
\item[\emph{\textbf{(a)}}]\quad
$\{\overline{G};\delta_{2},...,\delta_{r}\}$,
\vspace{2mm}
\item[\emph{\textbf{(b)}}]\quad
$ \{\overline{G}_{1:r},...,\overline{G}_{r:r}\}$
\end{itemize}
convey the same information about the joint distribution of $T_{1}$,...,$T_{r}$ and that they can be then recovered one from the
other as shown by the formula~\eqref{OrdStatEtDiagonals} and Proposition~\ref{prop:delta-da-G-k}.

We now restrict attention on the absolutely continuous case, where the joint
distribution of $T_{1}$,...,$T_{r}$ can be described in terms
of the corresponding m.c.h.r.\@ functions. In this respect, in terms of those
functions, we aim to single out characteristics of the joint distribution,
whose knowledge may be equivalent to that of the systems of functions in~\emph{\textbf{(a)}} and \emph{\textbf{(b)}}.

It is convenient on this purpose to fix attention on the marginal
distributions, of the different dimensions, for the random vector $(T_{1},...,T_{r})$. The following simple remark has a central role for
our task.

As already observed  (see Eq.~\eqref{LawOfMinimum}) the m.c.h.r.\@
functions $\lambda_{1|\emptyset}(t),...,\lambda_{r|\emptyset}(t)$ are strictly
related to the marginal laws  
of the minimal order statistics, indeed
\begin{equation}
		 \overline{G}_{1:r}(t)=\mathbb{P}(T_{1:r}>t)
				 =\exp\Big\{-\int_{0}^{t}\Lambda_{\emptyset}(s)\,ds\Big\},  \quad t>0,
		 \label{SurvFunctOfMinimumIDDD}
\end{equation}
where $\Lambda_\emptyset(t)=\sum_{j=1}^{r}\lambda_{j|\emptyset}(t)$.
Similarly (see Eq.s~\eqref{DefmchrD-J} and~\eqref{LawOfMinimum-A})  the m.c.h.r.\@ functions~$\lambda^A_{j|\emptyset}(t)$, $j\in A$,  are  related to the law of the minimum on an arbitrary set   $A\subset[r]$, indeed
\begin{equation}\label{Prob-T1:A>t}
	\mathbb{P}(T_{1:A}>t)
	=\exp\left\{-\int_{0}^{t}\Lambda^A_{\emptyset}(s)\,ds\right\},\quad t>0,
\end{equation}
where
$
\Lambda^A_{\emptyset}(t)=\sum_{j\in A}\lambda^A_{j|\emptyset}(t)
$
is the  one-dimensional failure rate of $T_{1:A}$. Concerning with this notation, observe that the failure rate  $ \Lambda^{[r]}_\emptyset(t)$ coincides with $\Lambda_\emptyset(t)$.
 \\

By the assumption of minimal stability we may restrict attention on  the subset of
variables $T_{1},...,T_{d}$, which are minimally stable, as well. Indeed, since $\mathbb{P}(T_{1:A}>t)=\mathbb{P}(T_{1:B}>t)$ for any $A,B\subseteq[r]$ such that $|A|=|B|$, then for any~$d=1,...,r$,
\begin{equation}\label{LAMBDA-d}
\Lambda^A_{\emptyset}(t)=\Lambda^{[d]}_\emptyset(t)=\sum_{j=1}^d\lambda^{[d]}_{j|\emptyset}(t),\quad  t\geq 0,  \quad
 \forall\, A\subset[r], \; \text{with}\; |A|=d, \end{equation}
 and therefore
\begin{equation}\label{Prob-T1:A>t=Prob-T1:d}
	\mathbb{P}(T_{1:A}>t)
	=\exp\left\{-\int_{0}^{t}\Lambda^{[d]}_{\emptyset}(s)\,ds\right\},\quad t>0, \quad \forall\, A\subset[r], \; \text{with}\; |A|=d.
\end{equation}

It will emerge that  the information contained in the systems \emph{\textbf{(a)}}  and \emph{\textbf{(b)}} is equivalent to the  knowledge embedded in the systems of functions defined by
\begin{itemize}
\item[\emph{\textbf{(c)}}]\quad
$\{\Lambda^{[1]}_\emptyset,...,\Lambda^{[r]}_\emptyset\}$.
\end{itemize}

Such equivalence is demonstrated by the relations respectively tying   \emph{\textbf{(c)}} with \emph{\textbf{(a)}}
and  \emph{\textbf{(c)}} with \emph{\textbf{(b)}}. Such relations will be detailed
below by means of the following Propositions~\ref{Talamone2-min} and~\ref{Circeo-min}. More precisely, in Proposition~\ref{Talamone2-min} we express \emph{\textbf{(a)}} and \emph{\textbf{(b)}} in terms of \emph{\textbf{(c)}}, whereas in Proposition~\ref{Circeo-min} \emph{\textbf{(c)}} is written in terms of \emph{\textbf{(a)}} and  \emph{\textbf{(b)}}.

	\begin{proposition}\label{Talamone2-min}
	Let the joint distribution of the minimally stable random lifetimes $T_{1},...,T_{r}$ be absolutely
	continuous, with
	$	\Lambda^{[d]}_\emptyset(t),$ for $d=1,...,r.$
	Then
	\begin{description}
\item[\emph{\textbf{(i)}}] \;  for any $t >0$
	\begin{equation}\label{barra-G-con.LAMBA-1}
	\overline{G}(t)=\exp\left\{-\int_0^t \Lambda^{[1]}_\emptyset(s)\, ds \right\},
 \end{equation}
 and furthermore, for $d=2,...,r$ and for any $u\in\left[  0,1\right]$,
\begin{equation}
	\delta_{d}(u)=\exp\left\{ -\int_{0}^{\overline{G}^{-1}(u)}\Lambda^{[d]}_\emptyset(s)\,ds\right\}; \label{CapDAil-min}
\end{equation}
\item[\emph{\textbf{(ii)}}]\;for $\ell=1,...,r$ and for any $t>0$,
\begin{equation}
	\overline{G}_{\ell:r}(t)=\sum_{h=r-\ell+1}^{r}\left(  -1\right)  ^{h-r-1+\ell}
	\binom{r}{h}\binom{h-1}{r-\ell}\exp\left\{  -\int_{0}^{t}\Lambda^{[h]}_\emptyset(s)\,ds\right\}. \label{CapDAntibes-min}
\end{equation}
\end{description}
\end{proposition}

\begin{proof}[Proof of \textbf{(i)}]	
Due	to minimal stability the random variables $T_i$  share the  marginal survival function with  $T_1$,
i.e.,
$$
\exp\left\{ -\int_0^t \lambda_{i|\emptyset}(s)\,ds\right\}=\exp\left\{ -\int_0^t \lambda_{1|\emptyset}(s)\,ds\right\}, \quad t>0.
$$
Therefore $\lambda_{i|\emptyset}(t)=\lambda_{1|\emptyset}(t)\equiv \Lambda^{[1]}_\emptyset(t)$, for any $t>0$, and~\eqref{barra-G-con.LAMBA-1} follows.
\\

As already observed, we may concentrate attention on the	random variables $T_{1},...,T_{d}$. Therefore to prove~\eqref{CapDAil-min},  on the one hand one has
	\[
	\mathbb{P(}\min_{i=1,...,d}T_{i}>t)=\exp\left\{-\int_{0}^{t}\Lambda^{[d]}_\emptyset(s)\,ds\right\}.
	\]
	On the other hand, taking into account~\eqref{SantaSevera} we can also write%
	\[
	\mathbb{P(}\min_{i=1,...,d}T_{i}>t)=\delta_{d}(\overline{G}(t)).
	\]	
	Thus Eq.~\eqref{CapDAil-min} is immediately achieved by comparing the preceding 	two formulas.
\\

\noindent \emph{Proof of \textbf{(ii)}. }
Taking into account Proposition~\ref{prop:min-stable-Ord-Stat},	 Eq.~\eqref{CapDAntibes-min} is immediately achieved by combining Eq.~\eqref{CapDAil-min} with Eq.~\eqref{OrdStatEtDiagonals}.
\end{proof}

\bigskip

\begin{proposition}\label{Circeo-min}Let the joint distribution of the minimally stable $T_{1},...,T_{r}$ be absolutely
	continuous.
	\\
	\begin{description}
\item[\emph{\textbf{(i)}}] \;
	 If the diagonal sections~$\delta_{2},...,\delta_{r}$ and
	the one-dimensional marginal survival function~$\overline{G}$   are given, then
\begin{equation}
	\Lambda^{[\ell]}_\emptyset(t)=-\frac{d}{dt}\log\left[  \delta
	_{\ell}(\overline{G}(t))\right],  \label{Sorrento-min}%
\end{equation}
   for $\ell=1,2,...,r$ and any $t> 0$.
		\\
\item[\emph{\textbf{(ii)}}] \;		
If  the marginal survival functions $\overline{G}_{1:r}
	,...,\overline{G}_{r:r}$ for the order statistics are given,   
respectively	denoting by $g_{1:r},...,g_{r:r}$ the probability density of the order statistics $T_{1:r},...,T_{r:r}$, then
\begin{align}\notag
	\Lambda^{[\ell]}_\emptyset(t)&=
\frac{\sum_{h=\ell}^{r}   (h)_\ell\, \big(g_{r-h+1:r}(t)- g_{r-h:r}(t)\big)
}{\sum_{h=\ell}^{r}  (h)_\ell\, \Big(\overline{G}_{r-h+1:r}(t)- \overline{G}_{r-h:r}(t)\Big)}
\\
&=	\frac{\sum_{k=1}^{r-(\ell-1)}(r-k)_{\ell-1}\, g_{k:r}(t)
}{\sum_{k=1}^{r-(\ell-1)} (r-k)_{\ell-1} \, \overline{G}_{k:r}(t)},
	\label{Palinuro-min}%
\end{align}
for $\ell=1,...,r$ and any $t>0$.
\end{description}
\end{proposition}

\begin{proof}[Proof of \textbf{(i)}]
  For $\ell=1$ Eq.~\eqref{Sorrento-min}  follows immediately by Eq.~\eqref{barra-G-con.LAMBA-1}. For $\ell=2,...,r$,  Eq.~\eqref{Sorrento-min} is immediately achieved by inverting the Eq.~\eqref{CapDAil-min}.
\\

\noindent \emph{Proof of \textbf{(ii)}.}
Eq.~\eqref{Palinuro-min} is obtained by resorting to 
 Eq.~\eqref{Sorrento-min}
and Proposition~\ref{prop:delta-da-G-k},  together with the fact that $\mathbb{P}\big( T_{1:A}>t\big)=\delta_d\big( \overline{G}(t)\big)$.
\end{proof}

These results can be further specialized to the  case of exchangeable times $T_1,...,T_r$.

For any $d=1,...,r$,  the exchangeable random lifetimes $T_1,...,T_d$,  are characterized  by the m.c.h.r.\@ functions
$\mu^{[d]}(t|k;t_1,..,t_k)$
for $k=0,...,d-1$, $0<t_1<\cdots<t_k\leq t$, (see~\eqref{ExchMCHR} and Proposition~\ref{prop:caratt-Scambiabili} for $r=d$),  and therefore
\begin{equation}\label{d-mu[d]}
\Lambda^{[d]}_\emptyset(t)= d \mu^{[d]}(t|0).
\end{equation}
In view of this remark,  the system of functions \textbf{\emph{(c)}} is equivalent to the  set of functions
\begin{itemize}
\item[\textbf{\emph{(c')}}]\quad
$
\{\mu^{[1]}(t|0),...,\mu^{[r]}(t|0)\},
$
\end{itemize}
which is therefore equivalent also to the systems of functions  \textbf{\emph{(a)}} and \textbf{\emph{(b)}}.
\\

More precisely, when $T_1,...,T_r$ are exchangeable lifetimes, then, with the above notation, Eq.s~\eqref{barra-G-con.LAMBA-1},~\eqref{CapDAil-min} and~\eqref{CapDAntibes-min} can be rewritten as
	\begin{equation}\label{cor:barra-G-con.LAMBA-1}
	\overline{G}(t)=\exp\left\{-\int_0^t \mu^{[1]}(s|0)\, ds \right\},  \quad t>0,
 \end{equation}
\begin{equation}
	\delta_{d}(u)=\exp\left\{ -d\,\int_{0}^{\overline{G}^{-1}(u)}\mu^{[d]}(s|0)\,ds\right\}, \quad u\in [0,1], \; d=2,...,r, \label{cor:CapDAil-min}%
\end{equation}
and
\begin{equation}
	\overline{G}_{\ell:r}(t)=\sum_{h=r-\ell+1}^{r}\left(  -1\right)  ^{h-r-1+\ell}%
	\binom{r}{h}\binom{h-1}{r-\ell}\exp\left\{  -h\,\int_{0}^{t}\mu^{[h]}(s|0)\,ds\right\}. \label{in-cor:CapDAntibes-min}%
\end{equation}

Similarly Eq.s~\eqref{Sorrento-min} and~\eqref{Palinuro-min} take the form
\begin{equation}
	\mu^{[\ell]}(t|0)=-\frac{1}{\ell}\,\frac{d}{dt}\log\left[  \delta
	_{\ell}(\overline{G}(t))\right], \quad t> 0, \quad  \ell=1,2,...,r, \label{cor:Sorrento-min}%
\end{equation}
		
\begin{align}\notag
	\mu^{[\ell]}(t|0)&=
\frac{1}{\ell}\,\frac{\sum_{h=\ell}^{r}   (h)_\ell\, \big(g_{r-h+1:r}(t)- g_{r-h:r}(t)\big)
}{\sum_{h=\ell}^{r}  (h)_\ell\, \Big(\overline{G}_{r-h+1:r}(t)- \overline{G}_{r-h:r}(t)\Big)}
\\
&=	\frac{1}{\ell}\, \frac{\sum_{k=1}^{r-(\ell-1)}(r-k)_{\ell-1}\, g_{k:r}(t)
}{\sum_{k=1}^{r-(\ell-1)} (r-k)_{\ell-1} \, \overline{G}_{k:r}(t)}, \quad t>0, \quad \ell=1,...,r.
	\label{cor:Palinuro-min}%
\end{align}


\section{Special cases}\label{sec:special}

The arguments developed in the previous sections will now be illustrated by
considering the two remarkable classes of models respectively defined by
Archimedean copulas and by multivariate hazard rate functions satisfying the
load-sharing condition. These choices in a sense correspond to the simplest
possible forms admitted in the two types of descriptions of a joint
distribution for lifetimes, respectively.

In particular, the analysis of these classes will  allow us to obtain some examples of
application for some of the results derived above, by showing the special form
taken by related formulas.
The arguments in Section~\ref{subsec:THLS}  also permit to pave the way for a better
understanding of the differences between the cases when lifetimes are
exchangeable or minimally stable and between non-order dependent and strictly
order dependent load sharing models. Furthermore we will be in a position to present
some heuristic ideas at the basis of the construction of minimally stable, but
non-exchangeable, multivariate models.

\bigskip

\subsection{Archimedean Copulas}${}$ \label{subsec:Archimedean}\\

 Let us consider  the case when the survival copula~$K$ is Archimedean with generator~$\psi$,
i.e., when $K=C_\psi$, with
$$
C_\psi(u_1,...,u_r):=\psi^{-1}\big(\psi(u_1)+\cdots +\psi(u_r) \big).
$$
We recall that the inverse function $\psi^{-1}$ has to be $r$-monotonic (see Theorem~6.3.6 in Schweizer and Sklar (1983)~\cite{Schweizer-Sklar-1983}, see also Nelsen~(2006~\cite{Nelsen-2006}).
 If  $\psi$ is also a strict generator, i.e., besides being a decreasing, convex function such $\psi(1)=0$,  it is also such that $\psi(0^+)=\infty$,
 then by definition (see Eq.~\eqref{def:delta-ell}) the sub-diagonals  assume the form
$$
\delta_\ell(u)=\psi^{-1}(\ell\psi(u)), \quad 2\leq \ell\leq r.
$$

When the common marginal survival function  $\overline{G}$ is given, then $T_1,...,T_r$ is an exchangeable model,
and   by~\eqref{OrdStatEtDiagonals} we get
$$
\overline{G}_{\ell:n}(t)=\sum_{h=r-\ell+1}^{r}(-1)^{h-r-1+\ell}\binom{r}{h}\,\binom{h-1}{r-\ell}\,\psi^{-1}(h\psi(\overline{G}(t))), \quad 1\leq \ell \leq r.
$$

Furthermore, by Eq.~\eqref{cor:Sorrento-min} one can get the m.c.h.r.\@ functions
\begin{align}\label{mu-ell-t-0-Archimedean}
\mu^{[\ell]}(t|0)&=- \frac{1}{\ell}\,\frac{d}{dt}\log\left[\psi^{-1}\big(\ell \psi (\overline{G}(t))\big)\right]
\\ & =  \frac{1}{\psi^\prime\Big(\psi^{-1}\big(\ell\psi ( \overline{G}( t))\big)\Big)} \, \psi^\prime \big(\overline{G}( t) \big) \, g( t)
\end{align}

Conversely, when the survival distribution functions  $\overline{G}_{k:r}$ are given
then by Corollary~\ref{cor:delta-da-G-k}
we immediately get the expression of $\psi^{-1}(d\psi(u))$ for any $d=1,...,r$.
In particular, when $d=r$ then
$$
\psi^{-1}(r\psi(u))=\delta_r(u)= \overline{G}_{1:r}\big(\overline{G}^{-1}(u)\big)
$$
and $d=r-1$
\begin{align*}
\psi^{-1}((r-1)\psi(u))=\delta_{r-1}(u)&=
\frac{1}{r} \overline{G}_{2:r}\big(\overline{G}^{-1}(u)\big)+\left(1-\frac{1}{r} \right)\delta_r(u)
\end{align*}

Concerning the problem of identifying the generator $\psi$ starting from the knowledge of the survival functions $\overline{G}_{k:r}$, the following remark can be of help.

\begin{remark}\label{rem:on-uniqueness-of-psi}
It is interesting to point out (see Jaworski~(2009) \cite{Jaworski-2009}) that  in general,  when $r>2$,  there exist infinite generators with the same diagonal section $\delta_r$, wheras    the generator $\psi$ is uniquely determined (up to a  multiplicative constant)  by  the pair $\delta_r$ and~$\delta_{r-1}$.   Unfortunately, the proof of the latter property is not  constructive.  However, when the diagonal $\delta_r$ satisfy the condition $\delta^\prime_r(1^-)=r$, then  Erderly et al.~(2013) \cite{Erdely-et-al-2013}  show that~$\psi$ is uniquely determined (up to a multiplicative constant) by $\delta_r$:
$$
\psi(u)\propto \lim_{m\rightarrow \infty} r^m \big( 1- \delta_r^{-m}(u)\big),
$$
where $\delta_r^{-m}$ is the composition of $\delta_r^{-1}$ with itself $m$ times.
\end{remark}

In particular, in the Schur-constant case, i.e., when
$\overline{G}=\psi^{-1}$, the sub-diagonals $\delta_h$ and the survival functions $\overline{G}_{\ell:n}$
 are determined by the marginal survival function $\overline{G}(t)$:
  $$
  \delta_h(\overline{G}(t))=\psi^{-1}(h\psi(\overline{G}(t)))=\overline{G}(ht)
  $$
$$
\overline{G}_{\ell:n}(t)=\sum_{h=r-\ell+1}^{r}(-1)^{h-r-1+\ell}\binom{r}{h}\,\binom{h-1}{r-\ell}\,\overline{G}(ht), \quad 1\leq \ell \leq r.
$$

Furthemore, in view of the particularly simple form of $\delta_{\ell},$
for $\ell=1,2,...,r,$ Eq.~\eqref{cor:Sorrento-min}
can be rewritten as
\begin{eqnarray}\label{cor:Sorrento-SchurC}
\mu^{[\ell]}(t|0) & = & -\frac{1}{\ell}\,\frac{d}{dt}\log\left[\overline{G}(\ell t)\right]
=\frac{g(\ell t)}{\overline{G}(\ell t)}.
\end{eqnarray}
Note that  for Schur-constant models one could get the above results also directly, taking into account that
$$
\mathbb{P}(T_1 >t_1,..., T_r>t_r)= \overline{G}(t_1+\cdots +t_r),
$$
and therefore
$$
\mathbb{P}(T_1 >t,..., T_\ell>t)=\mathbb{P}(T_1 >t,..., T_\ell>t, T_{\ell+1}>0,...,T_{r}>0)= \overline{G}(\ell\, t).
$$
\begin{example}
Let us consider the Archimedean model with
$$
\psi(u)=  (u^{-\alpha}-1)^{\frac{1}{\beta}},\quad \alpha>0, \, \beta\geq 1
$$
and  $\overline{G}(t)=e^{-t}$.
The inverse function $\psi^{-1}(t)=\frac{1}{(t^\beta+1)^\alpha}$ is a
completely monotonic function, so that
$C_\psi$ is a copula fon any $r\geq 2$.
Then, for any $A\subset [r]$ with $|A|=\ell$, one has
\begin{align*}
\mathbb{P}(T_{1:A}>t)&=\psi^{-1}(\ell\psi (\overline{G}(t)))
\\&=\frac{1}{\left( \left(\ell (e^{t\alpha}-1)^{\frac{1}{\beta}}\right)^\beta+1\right)^\alpha}= \frac{1}{\left(\ell^{\beta} e^{t\alpha}-\ell^{\beta}+1\right)^\alpha}.
\end{align*}
Therefore by \eqref{mu-ell-t-0-Archimedean} and  taking into account that
$$
\psi^\prime(u)=- \frac{\alpha}{\beta} \,
u^{-(\alpha+1)}\, \left( u^{-\alpha}-1\right)^{\frac{1}{\beta}-1},
$$
we get
\begin{align*}
\mu^{[\ell]}(t|0)&=
\frac{\,e^{\alpha t}\, \big(e^{\alpha t}- 1\big)^{\frac{1}{\beta}-1}
}{\left(\ell^{\beta} e^{t\alpha}-\ell^{\beta}+1\right)^{\alpha^2+\alpha}\, \left(\left(\ell^{\beta} e^{t\alpha}-\ell^{\beta}+1\right)^{\alpha^2}-1 \right)^{\frac{1}{\beta}-1}
}.
\end{align*}
In the Schur-constant case with the same generator, i.e., with
$$
\overline{G}(t)=\psi^{-1}(t)=\frac{1}{(t^\beta+1)^\alpha}, \quad g(t)=\alpha\, \beta \frac{t^{\beta-1}}{(t^\beta+1)^{\alpha+1}}
$$
by Eq.~\eqref{cor:Sorrento-SchurC} we get
$$
\mathbb{P}(T_{1:A}>t)=\frac{1}{\big((\ell t)^\beta+1\big)^\alpha},\quad \mu^{[\ell]}(t|0)=\alpha\, \beta \frac{(\ell t)^{\beta-1}}{(\ell t)^\beta+1}.
$$
\end{example}
\bigskip

\subsection{Time homogeneous load-sharing models}\label{subsec:THLS}

Load sharing models are characterized by the condition that the m.c.h.r.\@ functions  depend on current time and on the set  of
failed components at the current time, but do not depend on the   failure time, according to the following definition.

\begin{definition} [Load Sharing Models]
The joint distribution of the random variables $T_1,...,T_r$ is  a  \emph{Load Sharing model}
if it is absolutely continuous, and the m.c.h.r.\@ functions  depend neither on the  failure times nor on the order of failure, i.e.,
for any $k=0,1,...,r-1$,  there exist $\binom{r}{k}\, (r-k)$  functions $v\mapsto\lambda_{j|\{j_1,...,j_k\}}(v)$ such that, for any
 $0< v_1 <\cdots <v_k<v$,
and $(j_1,...,j_k)\in \Pi_k([r])$
$$
\lambda_{j|j_1,...,j_k}(v|v_{1},\cdots,v_{k})= \lambda_{j|\{j_1,...,j_k\}}(v).
$$
When the functions $v\mapsto\lambda_{j|\{j_1,...,j_k\}}(v)=\lambda_{j|\{j_1,...,j_k\}}$ are constant w.r.t.\@ time $v$, then the model $T_1,...,T_r$ is said a \emph{time homogeneous load sharing model} (THLS)
\end{definition}

Load Sharing models are well known and recurrently studied in the reliability literature (see, e.g., Spizzichino (2019) \cite{Spizzichino-2019-Rel},  Rychlik and Spizzichino (2021) \cite{Rychlik-Spizzichino-2021}, and the references cited therein). In particular the joint and marginal distributions of the order statistics has been studied in some details.
For what concerns the special case of exchangeability
see also Kamps (1995) \cite{Kamps-1995}.

 Concerning the above definition, notice that the m.c.h.r.\@ functions \--- which generally depends on the set of the failure units \--- are required to be independent of their failure times order.
It is interesting here to extend such a definition to a generalized class of models in which instead also the order of failure times may influence the m.c.h.r.\@ functions.

 Actually some of the existing results on LS models can be easily extended to this class.

\begin{definition} [order dependent Load Sharing Models]
The joint distribution of the random variables $T_1,...,T_r$ is  an  \emph{order dependent Load Sharing model}
if it is absolutely continuous,
 and the m.c.h.r.\@ functions do not depend on the  failure times, i.e., for any  $k=0,1,...,r-1$, there exist $\binom{r}{k} k! (r-k)$   functions $v\mapsto\lambda_{j|j_1,...,j_k}(v)$ such that, for any  $0< v_1 <\cdots <v_k<v$,
and $(j_1,...,j_k)\in \Pi_k([r])$
$$
\lambda_{j|j_1,...,j_k}(v|v_{1},\cdots,v_{k})= \lambda_{j|j_1,...,j_k}(v).
$$
When the functions $v\mapsto\lambda_{j|j_1,...,j_k}(v)=\lambda_{j|j_1,...,j_k}$ are constant w.r.t.\@ time $v$, then the model $T_1,...,T_r$ is said an \emph{order dependent time homogeneous load sharing model} (odTHLS)

A Load Sharing model is also an order dependent Load Sharing model.  We will say that a model is a \emph{strictly order dependent Load Sharing model} when
the m.c.h.r.\@ functions do depend on the order.
\end{definition}

 From an engineering-oriented viewpoint,
  strictly order-dependent load-sharing models do not
   seem very significant for applications in the field of reliability. However models with this property can come out in a natural way, from a mathematical standpoint. Furthermore,  they may emerge in different ways when showing the general interest of load-sharing models within the family of all the joint absolutely continuous probability distributions for lifetimes,
   as shown in De Santis, Spizzichino (2021) \cite{DeSantis-Spizzichino-2021} in the analysis of aggregation paradoxes. See also Example~\ref{example:r=3-uniform-frailty}, where    the condition of load sharing must be limited to strict order dependent THLS models on the purpose of finding minimally stable models which satisfy  specific symmetry properties without falling in the exchangeable case.

\bigskip

\noindent\textbf{Exchangeable THLS models.}\;
We first analyze the special case of exchangeability.
An exchangeable load sharing model clearly cannot be strictly order dependent, in that its m.c.h.r.\@ functions are such that
$\mu(t|k;t_{1},\ldots,t_{k})= \mu(t|k)$.
Furthermore it is time homogeneous if and only if  for any $k=0,1,..,r-1$ there exists a constant   $L(r-k)$  such that
\begin{align}\label{THLS-EX}
\mu(k)=\frac{L(r-k)}{r-k}, \; \forall\; j \in[r], \quad\text{and} \quad  \Lambda_{j_1,...,j_k}(t)= L(r-k).
\end{align}
In such a case it is easily seen  that
$$
\overline{G}_{k:r}(t)=\mathbb{P}\big( \tfrac{X_0}{L(r)}+\tfrac{X_1}{L(r-1)}+\cdots +\tfrac{X_{k-1}}{L(r-(k-1))}>t\big)
$$
where  $X_i\sim EXP(1)$, $i=0,1,2...,r-1$, are  indipendent random variables (see in particular Spiz\-zi\-chi\-no (2019) \cite{Spizzichino-2019-Rel},
Kamps (1995) \cite{Kamps-1995}, Cramer and Kamps (2003) \cite{Cramer-Kamps-2003}, and  references therein).
In other words, for any $k=1,..,r$, the distribution  of $T_{k:r}$ coincides with the distribution of  the sum of $k$ independent exponential distributions of parameters $\gamma_1=L(r),....,\gamma_k=L(r-(k-1))$. In the literature such a distribution is known  as Hyperexponential distribution (see Cramer and Kamps (2003) \cite{Cramer-Kamps-2003} and  references therein): for a fixed vector $\boldsymbol{\gamma}=(\gamma_1,...,\gamma_r)\in \mathbb{R}_+^r$,
$$
\overline{G}^{\boldsymbol{\gamma}}_k(t):=\mathbb{P}\Big( \sum_{j=1}^k \tfrac{Y_j}{\gamma_j}>t\Big)
$$
for $Y_1,...,Y_r$,  independent and standard exponential random variables.

When $\boldsymbol{\gamma}=(\gamma_1,...,\gamma_r)$ is
such that $\gamma_i\neq \gamma_j$ for all $i\neq j$,    the survival function and the probability density  are respectively given by
\begin{align*}
\overline{G}^{\boldsymbol{\gamma}}_k(t)&
 =
 \sum_{j=1}^k \left( \prod_{h\neq j}^{1,k} \frac{\gamma_h}{\gamma_h -\gamma_j}\right)\, e^{-\gamma_j t},
\end{align*}
and
$$
g^{\boldsymbol{\gamma}}_k(t)=
 \sum_{j=1}^k \left( \prod_{h\neq j}^{1,k} \frac{\gamma_h}{\gamma_h -\gamma_j}\right)\,\gamma_j e^{-\gamma_j t}.
$$

Therefore, denoting by $\boldsymbol{L}$ the vector
$$
\boldsymbol{L}:= \big(L(r),L(r-1),...,L(1) \big),
$$
we can write
$$
\overline{G}_{k:r}(t)= \overline{G}^{\boldsymbol{L}}_k(t).
$$
Furthermore, on the one hand formula~\eqref{bar-G-con-bar-G-k-r} takes the special form
\begin{align}\label{bar-G-EX-THLS}
\overline{G}(t)&=
\frac{1}{r} \sum_{k=1}^r \mathbb{P}\big( \tfrac{X_0}{L(r)}+\tfrac{X_1}{L(r-1)}+\cdots +\tfrac{X_{k-1}}{L(r-(k-1))}>t\big)
=
\frac{1}{r} \sum_{k=1}^r \overline{G}_k^{\boldsymbol{L}}(t).
\end{align}
On the other hand, taking into account~\eqref{prob-surv-min-on-J},,
for any $A\subset [r]$, with $|A|=d$, one has
\begin{align}\notag
\mathbb{P}\big(T_{1:A}>t\big)&=\delta_{d}\big(\overline{G}(t)\big)= \frac{d}{(r)_{d}}\,\sum_{k=1}^{r-d+1} (r-k)_{d-1} \,
\mathbb{P}\big( \tfrac{X_0}{L(r)}+\tfrac{X_1}{L(r-1)}+\cdots +\tfrac{X_{k-1}}{L(r-(k-1))}>t\big)
\\  \label{min-G-EX-THLS}
&=\frac{d}{(r)_{d}}\,\sum_{k=1}^{r-d+1} (r-k)_{d-1} \,
\overline{G}_k^{\boldsymbol{L}}(t)
\end{align}
and consequently, recalling the notation in~\eqref{d-mu[d]},~\eqref{cor:Palinuro-min} becomes
\begin{align}\label{mu-d-EX-THLS}
\mu^{[d]}(t|0)&=\frac{1}{d}\, \frac{\sum_{k=1}^{r-d+1} (r-k)_{d-1} \,
g_k^{\boldsymbol{L}}(t)}{\sum_{k=1}^{r-d+1} (r-k)_{d-1} \,
\overline{G}_k^{\boldsymbol{L}}(t)}
\end{align}

 In particular, assuming that $L(i)\neq L(j)$ for $i\neq j$, and setting
$$
\vartheta^{\boldsymbol{L}}_{\ell, k}:=  \prod_{h\neq \ell}^{0, k-1} \frac{L(r-h)}{L(r-h) -L(r-\ell)},
$$
  one has
\begin{align*}
\overline{G}(t)
&=
\frac{1}{r} \sum_{k=1}^r \sum_{j=1}^k \left( \prod_{h\neq j}^{1, k} \frac{L(r-(h-1))}{L(r-(h-1)) -L(r-(j-1))}\right)\, e^{-L(r-(j-1)) t}
\\&=
\frac{1}{r} \sum_{\ell=0}^{r-1} \left(\sum_{k=\ell+1}^{r} \vartheta^{\boldsymbol{L}}_{\ell, k}\right)\, e^{-L(r-\ell) t},
\\
\mathbb{P}\big(T_{1:A}>t\big)&=
 \frac{d}{(r)_{d}}\,\sum_{\ell=0}^{r-d}\left(\sum_{k=\ell+1}^{r-d+1} (r-k)_{d-1}\vartheta^{\boldsymbol{L}}_{\ell, k}\right)\, e^{-L(r-\ell) t}
\intertext{and}
\mu^{[d]}(t|0)&
=\frac{1}{d}\, \frac{\sum_{\ell=0}^{r-d}\left(\sum_{k=\ell+1}^{r-d+1} (r-k)_{d-1}\, \vartheta^{\boldsymbol{L}}_{\ell, k}\right)\;L(r-\ell) \,e^{-L(r-\ell) t}}{\sum_{\ell=0}^{r-d}\left(\sum_{k=\ell+1}^{r-d+1} (r-k)_{d-1}\vartheta^{\boldsymbol{L}}_{\ell, k}\right)\, e^{-L(r-\ell) t}}
\end{align*}

Note that, when $d=r$ then we  obviously get that the function $t \mapsto\mu^{[r]}(t|0)$ is constant and
$\mu^{[r]}(t|0)=\frac{1}{r}\,L(r)$, whereas for $d<r$ the function $t \mapsto\mu^{[d]}(t|0)$ is not constant. This fact is related to the circumstance that the $d$-dimensional marginal distributions of a load sharing model is generally not load sharing.

\bigskip

Before passing to the non-exchangeable case, we observe that in the present THLS exchangeable case
the functions~$\Psi(t;[r],\mathbf{j})$  defined in~\eqref{def:Psi-t-n-j} and~\eqref{Psi_xeqdiff} can be explicitly computed:
\begin{align}\notag
        \Psi(t;[r],\mathbf{j})&=\mathbb{P}\big(T_{j_1}<T_{j_2}<\cdots <T_{j_d}\leq t<T_i ,\; \forall i\notin\{j_1,...,j_d\} \big)
        \\\notag
      &= \frac{1}{d!} \, \mathbb{P}\big(T_{j}\leq t<T_i ,\; \forall\, j\in \{j_1,...,j_d\},\, \text{and}\, \forall\,i\notin\{j_1,...,j_d\} \big)
        \\\label{Psi-t-r-j-LS-EX}
        &= \frac{1}{d!}\,\frac{1}{\binom{r}{d}}\,\mathbb{P}\big(N(t)=d \big)= \frac{1}{(r)_d}\, \big[\overline{G}^{\boldsymbol{L}}_{d+1}(t)-\overline{G}^{\boldsymbol{L}}_{d}(t)\big]
        \end{align}
    Expression~\eqref{Psi-t-r-j-LS-EX}    turns out to be useful also in the analysis of minimally stable conditions.
Indeed for any odTHLS model
\begin{align}\notag
        \Psi(t;[r],\mathbf{j})&=\mathbb{P}\big(T_{j_1}<T_{j_2}<\cdots <T_{j_d}\leq t<T_i ,\; \forall i\notin\{j_1,...,j_d\} \big)
    \\    \notag  & =
        \prod_{\ell=1}^{d}\lambda_{j_{\ell}|j_1,....,j_{\ell-1}}
        \cdot   \int_{0}^{t}ds_d\int_{0}^{s_d}ds_{d-1}\cdots\int_{0}^{s_2}ds_{1}
        \\
        & {}\qquad \qquad \qquad \qquad \quad \;\left[
        e^{-(t-s_{d})\Lambda_{j_1,....,j_{d}}}
        \prod_{\ell=1}^{d}
        e^{-(s_{\ell}-s_{\ell-1})\Lambda_{j_1,....,j_{\ell-1}}}
        \right],\label{Psi-r-j-THLS-1}
        \intertext{and therefore, for $\boldsymbol{\Lambda}=\big(\Lambda_\emptyset,  \Lambda_{j_1},...,\Lambda_{j_1,...,j_d}\big)$, by~\eqref{Psi-t-r-j-LS-EX},  one gets} \Psi(t;[r],\mathbf{j})&=\prod_{\ell=1}^{d}\frac{\lambda_{j_{\ell}|j_1,....,j_{\ell-1}}}{\Lambda_{j_1,....,j_{\ell-1}}}\, \big[\overline{G}_{d+1}^{\boldsymbol{\Lambda}}(t)-\overline{G}_{d}^{\boldsymbol{\Lambda}}(t) \big],\label{Psi-r-j-THLS-LAMBDA}
\end{align}

\bigskip

Before continuing we  exhibit an example of non-exchangeable random lifetimes which are minimally stable.

\begin{example}\label{example-FABIO}
Let us take $r=3$ and consider the non-negative random variables  $T_{1}$, $T_{2}$, $T_{3}$ whose joint
distribution is given in terms of the  m.c.h.r.\@ functions as follows
\begin{align*}
&\lambda_{1|\emptyset}(t)=\lambda_{2|\emptyset}(t)=\lambda_{3|\emptyset}(t)=\frac{1}{3},
\\&
\lambda_{3|1}(t)=\gamma,\; \lambda_{2|1}(t)=1-\gamma,
\quad
\lambda_{1|2}(t)=\gamma,\; \lambda_{3|2}=1-\gamma
\quad\lambda_{2|3}(t)=\gamma,\;\lambda_{1|3}(t)=1-\gamma,
\intertext{for a fixed value  $\gamma\in\left(  \frac{1}{2},1\right) $
and finally}
&\lambda_{1|2,3}(t)=\lambda_{1|3,2}(t)=\lambda_{2|1,3}(t)=\lambda_{2|3,1}(t)=\lambda_{3|1,2}(t)=\lambda_{3|2,1}(t)=2.
\end{align*}

Thus  $T_{1},T_{2},T_{3}$ is a non-exchangeable THLS model: indeed, for instance,
one has
\[
\frac{1-\gamma}{3}=\mathbb{P}(T_{1}<T_{2}<T_{3})<\mathbb{P}(T_{1}<T_{3}<T_{2})=\frac{\gamma}{3} .
\]
The above inequality is implied by  the following observations:
\\
(i)\quad   for THLS models one has  (see, e.g., Spizzichino (2019) \cite{Spizzichino-2019-Rel})
$$
\mathbb{P}(T_{j_1}<T_{j_2}<T_{j_3})= \frac{\lambda_{j_1|\emptyset}}{\Lambda_\emptyset}\, \frac{\lambda_{j_2|j_1}}{\Lambda_{j_1}}\,\frac{\lambda_{j_3|j_1,j_2}}{\Lambda_{j_1, j_2}},
$$
\vspace{3mm}
\\ (ii) \quad
$
\Lambda_\emptyset=\Lambda_{j_1}=1, \quad \Lambda_{j_1,j_2}=2, \quad \text{for any $j_1, j_2 \in \{1,2,3\}$, $j_2\neq j_1$}.
$
\medskip

\noindent However the random variables $T_{1},T_{2},T_{3}$ are minimally stable.
Indeed by the previous observation (ii), and by Eq.~\eqref{Psi-r-j-THLS-1},  one has:
\\for any $j_1\in\{1,2,3\}$,
\begin{align*}
\Psi\big(t;[3],j_1\big)&= \mathbb{P}\big(T_{j_1}\leq t, T_{j}>t, j\neq j_1 \big)
\\
&= \int_0^t \lambda_{j_1|\emptyset} e^{-s\Lambda_{\emptyset}} \, e^{-(t-s)\Lambda_{j_1} }\, ds
=\int_0^t \frac{1}{3}\, e^{-s} \, e^{-(t-s)} \, ds= \frac{1}{3} \, t\, e^{-t};
\end{align*}
for any $(j_1,j_2)\in \Pi_2([3])$,
\begin{align*}
&\Psi\big(t;[3],(j_1,j_2)\big)= \mathbb{P}\big(T_{j_1}\leq T_{j_2}\leq t, T_{j_3}>t \big)
\\&= \int_0^t\, ds\,\int_0^s\, ds^\prime \, e^{-(t-s)\Lambda_{j_1,j_2}}\, \lambda_{j_1|\emptyset}\lambda_{j_2|j_1}\ e^{-s^\prime\Lambda_{\emptyset}} \, e^{-(s-s^\prime)\Lambda_{j_1}} \, ds
\\&= \frac{1}{3}\,\lambda_{j_2|j_1} \, e^{-2t} \,\int_0^t \, s\,e^{s} \, ds= \frac{1}{3}\, \lambda_{j_2|j_1} \, \big( e^{-t}t-e^{-t}+e^{-2t} \big).
\end{align*}
Taking into account that, for any $(j_1,j_2)$
$$
\lambda_{j_2|j_1} + \lambda_{j_1|j_2}=\gamma + (1-\gamma)=1
$$
we may apply Proposition~\ref{prop:caratt-min-stable}  to conclude that $T_1$, $T_2$, $T_3$ are minimally stable, and  that
$$
\mathbb{P}\big(T_{j_1}\leq t, T_{j_2}\leq t, T_{j_3}>t \big)= \frac{1}{3}\,  \big( e^{-t}t-e^{-t}+e^{-2t} \big), \quad \text{for any $(j_1,j_2,j_3)\in \Pi\big([3]\big)$.}
$$
\end{example}

\begin{example} \label{example-FABIO-BIS}
We again consider the model in the previous example, for which  $\Lambda_{\emptyset}=L(3)=1$, $\Lambda_{j_1}=L(2)=1$ and $\Lambda_{j_1, j_2}=L(1)=2$ (see (ii) therein), and compare it with
 the exchangeable THLS model \eqref{THLS-EX} such that
$$
L(3)=1, \; L(2)=1, \; L(1)=2, \quad  \text{or equivalently} \quad \mu(0)=\frac{1}{3},\; \mu(2)=\frac{1}{2},\; \mu(3)=2.
$$

 For the two models it turns out  that the family of the marginal survival functions of the order statistics coincide. We can see however that also their joint distributions
coincide. Actually the latter circumstance is a consequence of the condition that the functions $(j_1,...,j_k) \mapsto \Lambda_{j_1,...,j_k}$  are constant, only depending on $k$ (see condition \eqref{mathcal-L-singleton} in Example~\ref{example:matcal-L-singelton} and Remark~\ref{rem:mathcal-L-singleton} below).
\medskip

For  the minimally stable random variables $T_1$, $T_2$, $T_3$,  we now compute explicitly  the following survival functions
\begin{align*}
\mathbb{P}(T_{1}>t,T_{2}>t,T_{3}>t)&=e^{-t};
\\
\mathbb{P}(T_{1}>t, T_{2}>t)&= \mathbb{P}(T_1>t, T_2>t, T_3>t)+\mathbb{P}( T_{1}>t, T_{2}>t, T_{3}\leq t)
\\&= e^{-t} +  \frac{1}{3}\,t\, e^{-t}
= e^{-t}\,\Big(1+\frac{t}{3}\Big),
\\
\overline{G}(t)=\mathbb{P}(T_{1}>t) & =\mathbb{P}(T_1>t, T_2>t, T_3>t)+\mathbb{P}(T_{1}> t, T_{2}> t, T_{3}\leq t)
\\&{}\;+\mathbb{P}(T_{1}> t, T_{3}> t, T_{2}\leq t) +\mathbb{P}(T_{1}> t, T_{2}\leq t, T_{3}\leq t)
\\&= e^{-t}+ 2\,\frac{t}{3}\,e^{-t} +   \frac{1}{3}\,\big(t\,e^{-t}-e^{-t}+e^{-2t}\big)= \frac{2}{3}\, e^{-t} + t\, e^{-t} + \frac {1}{3} \, e^{-2t}
\end{align*}
Furthermore, taking into account that $\overline{G}_{1:r}(t)=e^{-\int_0^t\Lambda_\emptyset(s)\,ds}=\mathbb{P}(N(t)=0)$, and
$$
\overline{G}_{k:r}(t)=\mathbb{P}(N(t)\leq k-1), \quad \mathbb{P}(N(t)= h)= \binom{r}{h}\, \Psi\big(t;[r],(1,2,..,h)\big), \quad 1\leq h \leq k,
$$
we get
\begin{align*}
\overline{G}_{1:3}(t)=e^{-\Lambda_{\emptyset}t }=e^{-t},
\quad
\overline{G}_{2:3}(t)=e^{-t} \, (1+ t),
\quad
\overline{G}_{3:3}(t)=2\,e^{-t} \, t + e^{-2t}.
\end{align*}

\end{example}

\bigskip

\noindent\textbf{Minimally stable odTHLS models.\,}
We now pass to considering general properties of minimally stable  odTHLS models.  We start with a simple necessary condition for minimal stability.
\begin{lemma}\label{lemma:CN-ID+DD-THLS}
Let $T_1,...,T_r$ be an odTHLS model.   If  $T_1,...,T_r$ are  minimally stable then necessarily
$$
\lambda_{i|\emptyset}= \frac{\Lambda_{\emptyset}}{r}\quad \forall \, i \in [r]
$$
and
$$
\Lambda_{i}= \Lambda_1, \quad \forall \, i \in [r]
$$
\end{lemma}

\begin{proof}
By Proposition~\ref{prop:equivalent-properties}, we know that for $T_1,...,T_r$  minimally stable,   the probabilities
$\mathbb{P}(T_i\leq t, \,T_j>t, \forall j\neq i)$
necessarily assume the same value for any $i\in [r]$ and any time~$t$.
Taking into account that
$$
\mathbb{P}(T_i\leq t, \,T_j>t, \forall j\neq i)
=\lambda_{i|\emptyset} \int_0^t e^{-\Lambda_\emptyset s} \, e^{-\Lambda_i(t-s)} ds
$$
one immediately gets that
 $$
\mathbb{P}(T_i\leq t, \,T_j>t, \forall j\neq i)
=
\begin{cases}
\lambda_{i|\emptyset}\,\dfrac{ t\,e^{-\Lambda_\emptyset t}}{\Lambda_\emptyset} & \text{ if $\Lambda_i=\Lambda_\emptyset$.}
\\
&
\\
\lambda_{i|\emptyset}\,\dfrac{e^{-\Lambda_i t}- e^{-\Lambda_\emptyset t}}{\Lambda_\emptyset-\Lambda_i}\,
& \text{ if $\Lambda_i\neq \Lambda_\emptyset$,}
\end{cases}
$$
If there exist an $i_0\in [r]$ such that $\Lambda_{i_0}=\Lambda_\emptyset$ then necessarily $\Lambda_i=\Lambda_\emptyset$, for any $i\in [r]$. Consequently, also $\lambda_{i|\emptyset}=\lambda_{i_0|\emptyset}$.

Viceversa if there exist an $i_0\in [r]$ such that $\Lambda_{i_0}\neq \Lambda_\emptyset$, then necessarily $\Lambda_i\neq\Lambda_\emptyset$, for any $i\in [r]$. Furthermore one necessarily has
$$
\lambda_{i|\emptyset}\,\dfrac{e^{-\Lambda_i t}- e^{-\Lambda_\emptyset t}}{\Lambda_\emptyset-\Lambda_i}=\lambda_{1|\emptyset}\,\dfrac{e^{-\Lambda_1 t}- e^{-\Lambda_\emptyset t}}{\Lambda_\emptyset-\Lambda_1}.
$$
Then the thesis follows by the linear independence of the functions $t\mapsto e^{at}$ for different values of $a\in \mathbb{R}$.
\end{proof}

In the next result (see Proposition~\ref{prop:ID+DD-THLS-mixture} below) we show that the survival functions $\overline{G}_{h:r}$ of a minimally stable odTHLS model is a mixture of Hyperexponential distributions.

Before stating formally our result we need to introduce some notation.
Let $T_1,...,T_r$ be an ``order-dependent'' THLS.
For any permutation $\mathbf{j}=(j_1,...,j_r)\in \Pi[r]$ we will write $\boldsymbol{\Lambda}_{j_1,...,j_r}$ to denote the vector
$$
\boldsymbol{\Lambda}_{j_1,...,j_r}:=(\Lambda_\emptyset, \,\Lambda_{j_1}, ...,\Lambda_{j_1,...,j_{k-1}},..., \Lambda_{j_1,...,j_{r-1}}).
$$
We will also use the shorter notation $\boldsymbol{\Lambda}_{\mathbf{j}}$ instead of $\boldsymbol{\Lambda}_{j_1,...,j_r}$.

Then we consider the partition of $\Pi[r]$ generated by the equivalence relation
$$
\mathbf{j} \sim \mathbf{j'} \quad \Leftrightarrow \quad  \boldsymbol{\Lambda}_{j_1,...,j_r}= \boldsymbol{\Lambda}_{j'_1,...,j'_r}.
$$

If we define
\begin{equation}\label{mathcal-L}
\mathcal{L}:=\big\{\boldsymbol{L}: \; \exists\; \mathbf{j}\in \Pi([r])\; \text{ with }\; \boldsymbol{\Lambda}_{\mathbf{j}}=\boldsymbol{L} \big\}
\end{equation}
then each element of the partition may be labeled by the vectors $\boldsymbol{L}\in \mathcal{L}$:
$$
\Pi([r])= \bigcup_{\boldsymbol{L}\in \mathcal{L}} \Pi([r];\boldsymbol{L}),
$$
where
\begin{equation}\label{Pi-r-L}
 \Pi([r];\boldsymbol{L}):=\big\{ \mathbf{j}\in \Pi([r]) \; \text{such that}\; \boldsymbol{\Lambda}_{\mathbf{j}}=\boldsymbol{L} \big\}.
\end{equation}
Clearly the set $\mathcal{L}$ is always finite.

For our purposes it is convenient to label the coordinates of the vectors in $\mathcal{L}$ as follows:
$$
\boldsymbol{L}=\big(L(r), L(r-1),...,L(1)\big).
$$
With this position, in view of Lemma~\ref{lemma:CN-ID+DD-THLS},  when furthermore $T_1,...,T_r$ are minimally stable, then  $L(r)$ assumes the same value for any $\boldsymbol{L}\in \mathcal{L}$,  and the same happens for $L(r-1)$. More precisely one has
$$
L(r)=\Lambda_{\emptyset}, \quad L(r-1)=\Lambda_1= \Lambda_i, \; \forall \, i \in [r].
$$

\begin{proposition}\label{prop:ID+DD-THLS-mixture}
Let $T_1,...,T_r$ be an odTHLS model. Suppose that  the lifetimes $T_1,...,T_r$  are minimally stable. Then, with the notation introduced above, the survival functions $\overline{G}_{h:r}$, $h=1,...,r$, can be obtained as the following mixture of  Hyperexponential survival functions
\begin{equation}
\overline{G}_{h:r}(t)=\sum_{\boldsymbol{L}\in \mathcal{L}} \frac{|\Pi([r]; \boldsymbol{L})|}{r!} \, \overline{G}_{h:r}^{\boldsymbol{L}}(t).
\end{equation}
\end{proposition}
Note that
$$
\overline{G}_{k:r}^{\boldsymbol{L}}(t)= \mathbb{P}\big( \tfrac{X_0}{L(r)}+\tfrac{X_1}{L(r-1)}+\cdots +\tfrac{X_{k-1}}{L(r-(k-1))}>t\big), \quad k=1,...,r,
$$
is the survival function of the order statistics of an exchangeable THLS model with
$$
\mu(k)=\frac{L(r-k)}{r-k}.
$$
In the model of Example~\ref{example-FABIO}   the mixture turns out to be degenerate, since the set
$\mathcal{L}$ is the singleton $\{(1,1,2)\}$.  The latter fact explains the reason why  the family $\{\overline{G}_{1:3}, \overline{G}_{2:3},\overline{G}_{3:3}\}$ of that example coincides with the family of the exchangeable THLS model with
$\mu(0)=1/3$, $\mu(2)=1/2$, $\mu(3)=2$.

\begin{proof}[Proof of Proposition~\ref{prop:ID+DD-THLS-mixture}]
Consider a random permutation $\sigma_1,...,\sigma_r$, uniformly distributed in $\Pi([r])$.
Then  $T_{\sigma_1},..., T_{\sigma_r}$  is the symmetrized model  of  $T_1,...,T_r$, denoted by $\widetilde{T}_1,...,\widetilde{T}_r$ in Remark~\ref{rem:EX-con-delta=min-stable}.

 Clearly the joint distribution of $\widetilde{T}_1,...,\widetilde{T}_r$
is the mixture of the exchangeable THLS model with m.h.c.r.\@ functions
$$
\mu(k)=\frac{L(r-k)}{r-k},
$$
and with mixture weights given by $\frac{|\Pi([r]; \boldsymbol{L})|}{r!} $.
Therefore the marginal
  survival functions of the order statistics of the model  $\widetilde{T}_1$,..., $\widetilde{T}_r$ are given by
 $$
\sum_{\boldsymbol{L}\in \mathcal{L}} \frac{|\Pi([r]; \boldsymbol{L})|}{r!} \, \overline{G}_{h:r}^{\boldsymbol{L}}(t).
$$
 Finally the thesis follows   by Remark~\ref{rem:EX-con-delta=min-stable} and  the previous representation of the survival functions of the order statistics of an exchangeable THLS model.
\end{proof}

As a generalization of the arguments presented in Example~\ref{example-FABIO},  we characterize the set of all minimally stable odTHLS  models $T_1$,$T_2$, $T_3$ in terms of the  m.c.h.r.\@ functions.

\begin{example}[Minimally stable odTHLS with $r=3$]\label{example:r=3}

 Let us consider  the odTHLS model with  $T_{1}$, $T_{2}$, $T_{3}$ whose joint
distribution is given in terms of the  m.c.h.r.\@ functions $\lambda_{j_1|\emptyset}$, $\lambda_{j_2|j_1}$, $\lambda_{j_3|j_1,j_2}$, $j_1, j_2, j_3 \in \{1,2,3\}$, $j_2\neq j_1$, $j_3\neq j_1, j_2$.
		\\
We are going to prove that $T_{1}$, $T_{2}$, $T_{3}$ are minimally stable if and only if
conditions \emph{\textbf{(A1)}} and \emph{\textbf{(A2)}} below hold, together  with either  condition \emph{\textbf{(A3)}} or condition \emph{$\textbf{(A3)}^\prime$}, where
\begin{description}
\item[\textbf{(A1)}] there exists  a value $L(3)$ such that
$$
\lambda_{1|\emptyset}=\lambda_{2|\emptyset}=\lambda_{3|\emptyset}= \dfrac{L(3)}{3};
$$
\item[\textbf{(A2)}] there exists  a value $L(2)$ such that
$$
\Lambda_1=\lambda_{2|1}+\lambda_{3|1}=\Lambda_2=\lambda_{1|2}+\lambda_{3|2}=\Lambda_3=\lambda_{1|3}+\lambda_{2|3}=L(2);
$$
\item[\textbf{(A3)}]
there exists a value $L(1)$ such that
$$
\lambda_{j_3|j_1,j_2}=L(1), \quad \forall \, (j_1,j_2,j_3)\in \Pi([3]),
$$
and there exist two values $\gamma_1$ and $\gamma_2$ (possibly equal) such that
 \begin{equation}\label{gamma-1-2}
\{\lambda_{j_2|j_1},\lambda_{j_1|j_2}\}= \{\gamma_1, \gamma_2\}, \quad \text{for any $\{j_1,j_2\}$,}
\end{equation}
 and
 \begin{equation}\label{sum-gamma-1-2}
\gamma_1+\gamma_2=L(2);
\end{equation}
\item[$\textbf{(A3)}^\prime$]
 there exist two values $L^\prime(1)\neq L^{\prime\prime}(1)$
such that
$$
\lambda_{j_3|j_1,j_2}\in \big\{L^\prime(1),\, L^{\prime\prime}(1)\big\},\quad \forall\, (j_1,j_2,j_3)\in \Pi([3]),
$$
there exist two values $\gamma_1$ and $\gamma_2$ (possibly equal) such that \eqref{gamma-1-2} and \eqref{sum-gamma-1-2} hold,
and furthermore
$$
\lambda_{j_2|j_1}=\gamma_1,\; \lambda_{j_3|j_1,j_2}=L^\prime(1) \quad \Rightarrow \quad \lambda_{j_1|j_2}=\gamma_2,\;\lambda_{j_3|j_2,j_1}=L^{\prime\prime}(1)
$$
\end{description}
Note that the model is strictly order dependent only under condition $\emph{\textbf{(A3)}}^\prime$.
\\

Conditions  \emph{\textbf{(A1)}} and \emph{\textbf{(A2)}} are the necessary conditions of  Lemma~\ref{lemma:CN-ID+DD-THLS} with $L(3):=\Lambda_\emptyset$ and $L(2):=\Lambda_1$ and guarantee that
the following probabilities  take the same value for any $j_1\in\{1,2,3\}$,
$$
\Psi\big(t;[3],j_1\big)= \mathbb{P}\big(T_{j_1}\leq t, T_{j}>t, j\neq j_1 \big)= \int_0^t \lambda_{j_1|\emptyset} e^{-s\Lambda_{\emptyset}} \, e^{-(t-s)\Lambda_{j_1} }\, ds.
$$
Furthermore 
 for any $(j_1,j_2)\in \Pi_2([3])$, we get
\begin{align*}
\Psi\big(t;[3],(j_1,j_2)\big)&= \mathbb{P}\big(T_{j_1}\leq T_{j_2}\leq t, T_{j_3}>t \big)
\\&= \int_0^t\, ds\,\int_0^s\, ds^\prime \, e^{-(t-s)\Lambda_{j_1,j_2}}\, \lambda_{j_1|\emptyset}\lambda_{j_2|j_1}\, e^{-s^\prime\Lambda_{\emptyset}} \, e^{-\Lambda_{j_1}(s-s^\prime)} \, ds
\intertext{whereas}
\Psi\big(t;[3],(j_2,j_1)\big)&=  \int_0^t\, ds\,\int_0^s\, ds^\prime  \lambda_{j_2|\emptyset}\lambda_{j_1|j_2}\, e^{-\Lambda_{j_1,j_2}(t-s)}\, \ e^{-\Lambda_{\emptyset}s^\prime} \, e^{-\Lambda_{j_2}(s-s^\prime)}.
\end{align*}
Proposition~\ref{prop:caratt-min-stable} guarantees  that $T_1$, $T_2$, $T_3$ are minimally stable  if and only if the following probabilities assume the same value for any $t>0$ and for any $\{j_1,j_2\}\subset\{1,2,3\}$:
\begin{align*}
&\mathbb{P}\big(T_{j_1}\leq t, T_{j_2}\leq t, T_{j_3}>t \big)=\Psi\big(t;[3],(j_1,j_2)\big)+ \Psi\big(t;[3],(j_2,j_1)\big).
\end{align*}
Taking into account the necessary conditions \emph{\textbf{(A1)}} and \emph{\textbf{(A2)}}, and that when $r=3$, then $\Lambda_{j_1,j_2}=\lambda_{j_3|j_1,j_2}$ the previous condition is equivalent to require that
the following sums assume the same value, for any  $t>0$ and any $\{j_1,j_2\} \subset \{1,2,3\}$
\begin{align*}
&\lambda_{j_2|j_1}\,  \, \int_0^t\, ds\,\int_0^s\, ds^\prime \, e^{-\lambda_{j_3|j_1,j_2}(t-s)}\, \ e^{-L(3)s^\prime} \, e^{-L(2)(s-s^\prime)}
\\
&{}\qquad \quad+\lambda_{j_1|j_2}\,  \, \int_0^t\, ds\,\int_0^s\, ds^\prime \, e^{-\lambda_{j_3|j_2,j_1}(t-s)}\, \ e^{-L(3)s^\prime} \, e^{-L(2)(s-s^\prime)}.
\end{align*}

In its turn the above requirement is equivalent to
either condition \emph{\textbf{(A3)}} or $\emph{\textbf{(A3)}}^\prime$.
\end{example}

Among load sharing models, an interesting subclass is the class of the so-called \emph{uniform frailty models}, whose m.c.h.r.\@ functions are such that, for any $k=0,1,2,...,r-1$,
$$
\lambda_{j|j_1,...,j_{k}}=\frac{\Lambda_{j_1,...,j_{k}}}{r-k}, \quad \forall \, (j_1,...,j_{k})\in \Pi_k([r]).
$$
In the next example we analyze the condition of minimal stability for the model of the previous Example~\ref{example:r=3}, under the  additional condition of uniform frailty.

\begin{example}\label{example:r=3-uniform-frailty}
Let us consider the model $T_{1}$, $T_{2}$, $T_{3}$ of the previous Example~\ref{example:r=3}.
If besides the minimal stability condition, we impose the uniform frailty condition, then the model $T_{1}$, $T_{2}$, $T_{3}$ turns out to be non-exchangeable only if it is strictly order dependent.
Indeed  when condition \emph{\textbf{(A3)}} holds,  then   the additional uniform frailty  condition implies that $\gamma_1=\gamma_2=L(2)/2$ and therefore the model is exchangeable: for any $(j_1,j_2,j_3)\in \Pi([3])$
$$
\lambda_{j|\emptyset}=\frac{L(3)}{3},\quad \lambda_{j_2|j_1}=\frac{L(2)}{2}, \quad \lambda_{j_3|j_1,j_2}=L(1).
$$
On the contrary, when condition $\emph{\textbf{(A3)}}^\prime$ holds, the model is strictly order dependent.  Then the uniform frailty  and  the minimal stability conditions together  become: for any $(j_1,j_2,j_3)\in \Pi([3])$
$$
\lambda_{j_1|\emptyset}=\frac{L(3)}{3},  \quad  \lambda_{j_2|j_1}=\frac{L(2)}{2}, \quad  \big\{\lambda_{j_3|j_1,j_2}, \lambda_{j_3|j_2,j_1}\big\}=\big\{L^{\prime}(1),L^{\prime\prime}(1)\big\}.
$$
\end{example}

\begin{example}\label{example:matcal-L-singelton}
Let us
assume that $T_1,..,T_r$, with $r\geq 3$, is an odTHLS model described by the m.c.h.r.\@ functions $\lambda_{j|j_1,...,j_{d-1}}$, $d=1,...,r$, $(j_1,...,j_{d-1},j)\in \Pi([r])$ with the usual convention that when $d=1$ then $\lambda_{j|j_1,...,j_{d-1}}=\lambda_{j|\emptyset}$. We are going to
  characterize all the minimally stable odTHLS models  in the particular
 case when the set $\mathcal{L}$ is the singleton
$\big\{\big( L(r),L(r-1),...,L(1)\big) \big\}$, i.e.,
\begin{equation}\label{mathcal-L-singleton}
\Lambda_{j_1,...,j_k}=\sum_{j\notin \{j_1,...,j_k\}} \lambda_{j|j_1,...,j_k}=L(r-k), \quad \forall \, k=0,1,...,r-1.
\end{equation}
By Proposition~\ref{prop:ID+DD-THLS-mixture}, for all the above minimally stable odTHLS models,  one has  $\overline{G}_{k:r}=\overline{G}^L_{k}$, $k=1,2,..., r$, i.e., the same Hyperexponential  marginal survival functions  of the exchangeable THLS model~\eqref{THLS-EX}. 
 As a consequence, all  minimally stable models share with the latter exchangeable model also
 $(i)$ the  $1$-dimensional marginal survival function  of each $T_i$;
 $(ii)$  the survival function  of each $T_{i:A}$, which are respectively equal to \eqref{bar-G-EX-THLS} and \eqref{min-G-EX-THLS}. Furthermore they share also
  the rates $\Lambda_\emptyset^{[d]}(t)$, for any $d=1,2,...,r$, which are equal to 
 $
 \Lambda_\emptyset^{[d]}(t)= d \mu^{[d]}(t|0)
 $,
 where $\mu^{[d]}(t|0)$ are given in
 \eqref{mu-d-EX-THLS}.
\\

In this case the characterization of the condition that  $T_1,...,T_r$ are  minimally stable is a consequence of Proposition \ref{prop:caratt-min-stable} together with \eqref{Psi-r-j-THLS-LAMBDA}:
 the m.c.h.r.\@ functions satisfy the following system of equations
$$
\begin{cases}
\sum_{j\notin \{j_1,...,j_{d-1}\}}\lambda_{j|j_1,...,j_{d-1}}=L(r-(d-1)),&  {}\quad j_1 \in [r]
\\
&
\\
\sum_{\mathbf{j}\in \Pi(I)} \prod_{h=1}^d
\lambda_{j_{h}|j_{1},...,j_{h-1}}= \frac{1}{\binom{r}{d}}\,
\prod_{\ell=1}^{d}  L(r-(\ell-1)), & \text{for any $I\subset [r]$, with $|I|=d$,}
\\{}\qquad \qquad \qquad \qquad \text{$d=1,...,r$. } &
\end{cases}
$$
This characterization yields that there exist infinitely many minimally stable odTHLS models satisfying condition~\eqref{mathcal-L-singleton}. This statement is a consequence of  the observation
that one can see the previous system as a family of nested linear systems:
\begin{align*}
&
\begin{cases}
\sum_{j\in [r]} \lambda_{j|\emptyset}=L(r),&
\\
\lambda_{j|\emptyset}= \frac{1}{r} \, L(r);& {}\quad j\in [r]
\end{cases}
\intertext{once $\lambda_{j|\emptyset}$ are given, then $\lambda_{j_2|j_1}$ are the solutions $x_{j_1,j_2}$, $j_2\neq j_1$, of}
&\begin{cases}
\sum_{j\neq j_1} \lambda_{j_1|\emptyset}x_{j_1,j} 
=L(r-1),& {}\quad \forall\,j_1 \in [r]
\\
\lambda_{j_1|\emptyset}\, x_{j_1,j_2}+\lambda_{j_2|\emptyset}\, x_{j_2,j_1}
= \frac{1}{\binom{r}{2}} \, L(r)L(r-1);&  {}\quad \forall\,\{j_1,j_2\} \subset [r];
\end{cases}
\intertext{once $\lambda_{j|\emptyset}$ and $\lambda_{j|j_1}$ are given, then $\lambda_{j_3|j_1,j_2}$ are the solutions $x_{j_1,j_2,j_3}$ of}
&\begin{cases}
\sum_{j\notin \{j_1,j_2\}} x_{j_1,j_2,j}
=L(r-2),&  {}\quad \forall\, (j_1, j_2)\in \Pi_2([r]),
\\
\sum_{(j_1,j_2,j_3)\in \Pi(\{k_1,k_2,k_3\}) }\lambda_{j_1|\emptyset}\lambda_{j_2|j_1} x_{j_1,j_2,j_3}&
\\{}\qquad\qquad= \frac{1}{\binom{r}{3}} \, L(r)\,L(r-1) L(r-2);& {}\; \forall\, \{k_1,k_2,k_3\}\subset [r];
\end{cases}
\end{align*}
and so on.
Since the above nested systems always  admit the solutions $x_{j_1,...,j_k,j}=\lambda_{j|j_1,..,j_{k}}
=\frac{L(r-k)}{r-k}$,  then there are  infinite solutions.

\end{example}

\begin{remark}\label{rem:mathcal-L-singleton}
As we are going to show, the above condition~\eqref{mathcal-L-singleton}, i.e., the condition that $\mathcal{L}$ is a singleton, though does not imply minimal stability, it does imply the following condition: for any permutation $(j_1,...,j_r)\in \Pi([r])$,
\begin{equation}\label{weak-exchangeability}
\mathbb{P}\big(T_{k:r}>t | T_{j_1}< T_{j_2}<\cdots < T_{j_r}\big)=\mathbb{P}\big(T_{k:r}>t\big).
\end{equation}
Such a  condition emerges in a natural way even in more general settings beyond load-sharing, as pointed out in Navarro et al.~(2008)~\cite{Navarro-et-al-2008}, where it has been referred to as a condition of \emph{weak exchangeability} (see also Navarro et al.~(2021)~\cite{Navarro-Rychlik-Spizzichino-2021}).
\\ In the case of odTHLS models, condition~\eqref{mathcal-L-singleton} implies an even  stronger property:
\begin{equation}\label{stronger-weak-exchangeability}
\mathbb{P}\big(T_{k:r}>t_k, k=1,2,..,r | T_{j_1}< T_{j_2}<\cdots < T_{j_r}\big)=\mathbb{P}\big(T_{k:r}>t_k, k=1,2,...r\big).
\end{equation}
More precisely the
joint distribution of $\big(T_{1:r}, ..., T_{r:r} \big)$  coincides with the joint distribution of
$$
\left(\frac{Y_0}{L(r)}, \,\frac{Y_0}{L(r)}+\frac{Y_1}{L(r-1)}, \ldots , \, \frac{Y_0}{L(r)}+\frac{Y_1}{L(r-1)}+\cdots  \frac{Y_{r-1}}{L(1)}\right),
$$
where $Y_k$, $k=0,1,...,r-1$, are i.i.d.\@ standard exponential, i.e., the joint distribution of the order statistics of an exchengeable THLS model.

Indeed, one can easily extend Corollary~3 in~Rychlik and Spizzichino~(2021)~\cite{Rychlik-Spizzichino-2021} for THLS models, to odTHLS ones:
for any permutation $(j_1,...,j_r)\in \Pi([r])$,
the conditional joint distribution of $\big(T_{1:r}, ..., T_{r:r} \big)$  given the following event
$$
\big\{T_{j_1}< T_{j_2}<\cdots < T_{j_r}\big\}, \quad
$$
coincides with the law of
$$
\left(\frac{Y_0}{\Lambda_\emptyset}, \,\frac{Y_0}{\Lambda_\emptyset}+\frac{Y_1}{\Lambda_{j_1}}, \ldots , \, \frac{Y_0}{\Lambda_\emptyset}+\frac{Y_1}{\Lambda_{j_1}}+\cdots + \frac{Y_{r-1}}{\Lambda_{j_1,...,j_{r-1}}}\right).
$$

In the frame of load-sharing models, condition~\eqref{mathcal-L-singleton}  also emerges   in De Santis and Spizzichino~(2021)~\cite{DeSantis-Spizzichino-2021}, where it plays an important role for the special type of problems studied therein.
\end{remark}

\appendix
\section{}
\begin{example}[A procedure to construct a DD, not exchangeable, $n$-dimensional copulas]\label{example-DD-n-copulas}
Our aim is to prove a generalization of Example~\ref{example:C-3-construction}. We start by proving that, given a DD $n-1$-dimensional copula   $C_{n-1}$,
 then it is possible to construct a $n$-dimensional copula $C_n$ which is DD.
  Subsequently we show that by using this construction recursively, starting with a $2$-dimensional not symmetric copula, the copulas $C_n$ are not exchangeable.

Given the DD copula $C_{n-1}$,   we define a  $n$-dimensional copula~$C_n$ as
\begin{align*}
&C_{n}(u_1,...,u_{n-1}, u_n)
\\
:=& \frac{1}{n}\,\big[C_{n-1}(u_{1},...,u_{n-1})\cdot u_n+C_{n-1}(u_{2},...,u_{n-1}, u_{n})\cdot u_1+\cdots
\\&{}\quad \cdots + C_{n-1}(u_{n-1}, u_{n}, u_{1},...,u_{n-3})\cdot u_{n-2}+ C_{n-1}(u_{n},u_{1},...,u_{n-2}) \cdot u_{n-1}\big],
\end{align*}
Namely $C_{n}$ is obtained as the symmetric mixture of the copulas
 over the $n$   cyclic permutations of $(1,2,...,n)$
 $$
 \sigma_1=(1,2,...,n), \;\text{and}\; \sigma_k=(k,k+1,...,n, 1,..,k-1), \quad  2\leq k\leq n,
 $$
\begin{align}\label{def:C_n}
C_{n}(u_1,...,u_{n-1}, u_n):= \frac{1}{n}\, \sum_{k=1}^n  C_{n-1}(u_{\sigma_k(1)},...,u_{\sigma_k(n-1)})\cdot u_{\sigma_k(n)}.
\end{align}
It is easy to see that $C_n$ is DD, with diagonal sections
$$
\delta^{C_n}_{n}(u)=C_n(u,...,u,u)=\delta^{C_{n-1}}_{n-1}(u) \cdot u, \quad \delta^{(C_n)}_1(u)=u
$$
and 
\begin{equation}\label{eq:recursive-delta-d}
\delta^{C_n}_d(u)= \frac{d}{n}\,\delta^{C_{n-1}}_{d-1}(u)\cdot u+ \left(1- \frac{d}{n} \right)\, \delta^{C_{n-1}}_d(u), \quad  2\leq d\leq n-1.
\end{equation}
Indeed when $(u_1,...,u_{n-1}, u_n)=(u \mathbf{e}_A+ \mathbf{e}_{A^c})$, with $A\subset [n]$, and  $|A|=d$, then, for any $k=1,...,n$,
one can write  the $n-1$-dimensional vectors appearing in~\eqref{def:C_n} as
$$(u_{\sigma_k(1)},  u_{\sigma_k(2)},...,u_{\sigma_{k}(n-1)})\equiv (u_{k},u_{k+1},..., u_n,...,u_{k-2})=(u \mathbf{e}_{B}+ \mathbf{e}_{B^c}),
$$
 where  $B$   is a suitable subset of $[n-1]$. Furthermore the cardinality $|B|$  assumes either the values $d$  or the value $d-1$, depending on the value of  $u_{k-1}$, namely
$$
|B|=
\begin{cases} d& \text{ if  $u_{\sigma_k(n)}\equiv u_{k-1}=1$}\vspace{2mm}
\\ d-1 & \text{ if $u_{\sigma_k(n)}\equiv u_{k-1}=u$},
\end{cases}
$$
where we have used  the convention that $u_0=u_n$.
\\

Starting from~\eqref{eq:recursive-delta-d} one can easily prove that
$$
\delta^{C_n}_d(u)= \alpha_{n,d}\,u^d+(1-\alpha_{n,d})\,  C(u,u) u^{d-2} , \quad 1\leq d\leq n,
$$
with
$$
\alpha_{n,1}=1,  \quad
\alpha_{n,d}=\frac{d}{n} \alpha_{n-1,d-1}+ \Big(1-\frac{d}{n}\Big) \alpha_{n-1,d}.
$$

Starting  from a fixed permutation $(\pi(1),...,\pi(n))\notin \{\sigma_k,\; k=1,...,n\}$, 
a similar procedure can be used to construct another DD $n$-dimensional copula (possibly different from~$C_n$), by using the $n$ cyclic permutations $\sigma_k\in \Pi_{\mathcal{C}}(1,2,...,n)$: 
$$
C_{n, \pi}(u_1,...,u_{n-1},u_n):= \frac{1}{n} \, \sum_{k=1}^{n} C_{n-1}\big(u_{\pi(\sigma_{k}(1))}, u_{\pi(\sigma_{k}(2))},...,
u_{\pi(\sigma_{k}(n-1))}\big)\cdot u_{\pi(\sigma_{k}(n))}.
$$
When the procedure is implemented recursively starting with  a fixed permutation $\pi \in \Pi([n])$, and with a not-symmetric copula $C_2(u,v)=C(u,v)$, as in Example~\ref{example:C-3-construction},
denote by $U_1,...,U_n$ the random variables associated to the copula of $C_{n,\pi}$, $n\geq 3$. Then for any fixed $i=1,...,n$ the $2$-dimensional marginals of $U_i, U_{i+1}$  (with the convention that $U_{n+1}=U_0$) 
are obtained recursively as
$$
C_{n,\pi}(u ,v,\overbrace{1,...,1}^{n-2})= \frac{1}{n}\left[(n-2)C_{n-1,\pi}(u,v,\overbrace{1,...,1}^{n-3})+ 2uv\right], \quad u,v\in [0,1],
$$
and therefore are  all equal, i.e.,
$$
C_{n, \pi}\big(u \mathbf{e}_{\{i\}}+v\mathbf{e}_{\{i+1\}}+ \mathbf{e}_{[n]\setminus\{i,i+1\}}\big)= C_{n,\pi}(u,v,1,....,1).
$$
Notice that therefore
$$
C_{n,\pi}(u,v,1,....,1)\neq C_{n,\pi}(v,u,1,....,1),
$$
so that the copulas $C_{n,\pi}$ are not symmetric.
\end{example}

Since we are particularly interested in examples with absolutely continuous joint and marginal distributions, observe that if we start this procedure with a absolutely continuous copula we get absolutely continuous copulas.

Furthermore
 we recall the class of absolutely continuous examples given in Navarro and Fer\-nan\-dez-San\-chez (2020) \cite{Navarro-Fernandez-Sanchez-2020} (see  in particular Proposition~1 therein).
The class in~\cite{Navarro-Fernandez-Sanchez-2020} may be seen as a particular case of the larger class considered in the next example.

\begin{example}[Negative mixtures of  DD copulas are DD]
\label{example:D+C1-C2}
Suppose that $D(u_1,...,u_r)$ is an absolutely continuous exchangeable copula with  probability density $d$ such that
$$
0<\underline d \, \rho(u_1,...,u_r)\leq d(u_1,...,u_r),
$$
for some positive function $\rho$.
\\
Let $C_i(u_1,...,u_r)$, $i=1,2,$ be two different copulas which are DD, but non-exchangeable, and absolutely continuous, with probability density $c_i(u_1,...,u_r)$ such that
$$
0\leq c_i(u_1,...,u_r)\leq \overline{c}\,\rho(u_1,...,u_r).
$$
Assume also that the function $c_1(u_1,...,u_r)-c_2(u_1,...,u_r)$ is not symmetric, and
define
\begin{equation}\label{eq:K-alpha}
K_\alpha(u_1,...,u_r):= D(u_1,...,u_r)+\alpha\big[C_1(u_1,...,u_r)-C_2(u_1,...,u_r)\big].
\end{equation}
\medskip

If $\alpha$ is strictly positive and sufficiently small, then
 $K_\alpha$ is an absolutely continuous DD copula, but not exchangeable.
 \bigskip

 We now proceed with  the proof of the previous statement.
\\

The function $K_\alpha$ defined in~\eqref{eq:K-alpha} has a density
$$
k_\alpha(u_1,...,u_r)= d(u_1,...,u_r)+\alpha \big[c_1(u_1,...,u_r)- c_2(u_1,...,u_r)\big],
$$
such that the integral
$$
\int_{[0,1]^n}k_\alpha(u_1,...,u_r) du_1\cdots du_r=1.
$$
Therefore 
 $k_\alpha$ is a probability density,  
 if and only if $k_\alpha(u_1,...,u_r)\geq 0$  for any  $(u_1,...,u_r)\in (0,1)^r$. The condition $\overline{c}\alpha \leq \underline{d}$ implies that
$$
k_\alpha(u_1,...,u_r) \geq \underline{d} \rho(u_1,...,u_r) +\alpha [0-\overline{c}\, \rho(u_1,...,u_r)]\geq 0
$$
The assumption on the densities may be weakened, for instance, it is clearly not necessary any assumption on the density of $C_1$. Furthermore the example could be generalized to the case
$$
K_{\alpha_1,...,\alpha_m}(u_1,...,u_r):= D+\sum_{k=1}^m \alpha_k\, \big[C_{1,k}(u_1,...,u_r)-C_{2,k}(u_1,...,u_r)\big],
$$
with suitable conditions on $\alpha_i$ and $C_{i,k}$, for $k=1,...,m$, $i=1,2$.

Finally it is interesting to note that $K_\alpha$ is a negative mixture of copulas,    	and that negative mixture of i.i.d.\@ random variable  are linked to finite exchangeability, and the problem of extendibility.
\end{example}

\end{document}